\documentclass[final,3p,times]{elsarticle}
\usepackage{amsmath}
\usepackage{amssymb}
\usepackage[lined,boxruled,vlined,linesnumbered]{algorithm2e}
\usepackage{graphicx}
\usepackage{subfig}

\usepackage{pgf}
\usepackage{tikz}
\usepackage{color}
\usepackage[showonlyrefs]{mathtools}

\usetikzlibrary{positioning,arrows,automata}

\newcommand{\range}{\mbox{\rm range}}

\newcommand{\nullsp}{\mbox{\rm null}}
\newcommand{\mathR}{\mathbb{R}}
\newcommand{\mathC}{\mathbb{C}}
\newcommand{\diag}{\mbox{\rm diag}}
\newcommand{\one}{\mathbf{1}}
\newcommand{\fov}[1]{\mathcal{F}(#1)}

\newtheorem{theorem}{Theorem}
\newproof{proof}{Proof}

\journal{Applied Numerical Mathematics}

\begin{document}

\begin{frontmatter}

\title{An Adaptively Constructed Algebraic Multigrid Preconditioner for Irreducible Markov Chains \tnoteref{title_fn} \tnotetext[title_fn]{The work of the first author was partially supported by the National Science Foundation grant DMS-1320608 and by a U.S. Department of Energy subcontract B605152.}} 
\address[address2]{Department of Mathematics, Bergische Universit\"at Wuppertal, 42097 Germany}
\address[address1]{Department of Mathematics,
Pennsylvania State University, University Park, PA 16802, USA}

\author[address1]{James Brannick}
\author[address2]{Karsten Kahl} 
\author[address2]{Sonja Sokolovic} \ead{sokolovic@math.uni-wuppertal.de, +49 202 439 3776}

\begin{abstract}
The computation of stationary distributions of Markov chains is an important task
in the simulation of stochastic models. The linear systems arising in such applications involve
non-symmetric M-matrices, making algebraic multigrid
methods a natural choice for solving these systems.
In this paper we investigate extensions and
improvements of the bootstrap algebraic multigrid framework for solving these systems.
This is achieved by reworking the bootstrap setup process to use singular vectors
instead of eigenvectors in constructing interpolation and restriction.
We formulate a result concerning the convergence speed of GMRES for singular systems and experimentally justify why rapid convergence of the proposed method can be expected. We demonstrate its fast convergence and the favorable scaling behaviour for various test problems.


\end{abstract}



\end{frontmatter}


\section{Introduction}\label{sec:introduction}
Given the nonsymmetric transition matrix $A$ of a discrete
Markov chain, the task is to compute the steady state vector $x \neq  0$ of the Markov
chain such that
\begin{equation}\label{eq:Axx}
Ax = x,
\end{equation}
corresponding to the eigenvector of $A$ with eigenvalue 1.
If the underlying  Markov chain defining the matrix A is irreducible, then the existence and uniqueness (up to a scalar factor) of $x$ can be guaranteed by the Perron-Frobenius theorem (cf.~\cite{BermanPlemmons1994}). To compute the steady state vector $x$, we reformulate \eqref{eq:Axx} as a (singular) linear system of equations
\begin{equation}\label{eq:Bx0}
Bx = 0, \text{ where } B = I - A
\end{equation}
as has been proposed in, e.g.,~\cite{BoltenBrandtBrannickFrommerKahlLivshits2011,DeSterckManteuffelMcCormickRugeMillerPearsonSanders2010}.
We then adaptively construct an algebraic multigrid hierarchy for the matrix $B$ by means of the \emph{bootstrap algebraic multigrid framework} (BAMG)~\cite{BoltenBrandtBrannickFrommerKahlLivshits2011,BrandtBrannickKahlLivshits2011,Kahl2009,ManteuffelMcCormickParkRuge2010}. In contrast to the BAMG  approach considered in \cite{BoltenBrandtBrannickFrommerKahlLivshits2011}, which aimed at approximating eigenvectors of $A$ in the bootstrap setup, we now take left and right singular vectors of $A$ as test vectors for the construction of the transfer operators of the multigrid hierarchy.
This approach has been developed for nonsymmetric M matrices coming from discretization of convection diffusion
problems ~\cite{BrezinaManteuffelMcCormichRugeSanders2010} and has been shown to be effective
in practice for some problem classes. As we show here, such a strategy is in addition well suited for transition matrices coming from Markov chains.

In a second phase of the proposed method, 
the adaptive process is exited and the existing multigrid hierarchy is used to formulate a  preconditioner for a GMRES solver applied to the residual equation, computed from the approximation $x$ obtained in the first (setup) phase. 
Since the system matrix $B$ is singular, the convergence of the preconditioned GMRES iteration cannot be guaranteed in general, 
but we find that in practice the preconditioned method convergences for a wide variety of test problems.

While the algorithm we develop can be viewed as a general approach for solving nonsymmetric problems, we focus here on nonsymmetric and singular $M$-matrix systems arising in Markov chain applications. Important applications in which one needs the steady state solution of a Markov chain include statistical mechanics~\cite{Moyal1949}, queuing theory~\cite{Neuts1978}, analysis of telecommunication networks~\cite{KriegerMuller_ClostermannSczittnick1990}, information retrieval \cite{Benoit2005} and web ranking (e.g., Google \mbox{PageRank})~\cite{DeSterckManteuffelMcCormickNguyenRuge2008, LangvilleMeyer2006}. Many different methods for computing the steady state distribution have been proposed in the literature. These include the so called aggregation/disaggregation methods. These can be related to algebraic multigrid methods in the sense that they use a series of increasingly smaller Markov chains which are generated by aggregating groups of states into a single state. Most aggregation/disaggregation methods are two-level methods (cf.~\cite{Haviv1987,KouryMcAllisterStewart1984,Krieger1995}) as they are based on the original approach in~\cite{SimonAndo1961,Takahashi1975}, and it took approximately twenty years until multi-level aggregation/disaggregation methods were proposed~\cite{HortonLeutenegger1994,HortonLeutenegger1994_2}. Other, more recent, research on the design of multilevel approaches for Markov chains include the development of methods related to the \emph{smoothed aggregation multigrid} framework~\cite{DeSterckManteuffelMcCormickNguyenRuge2008,DeSterckManteuffelMcCormickRugeMillerPearsonSanders2010,DeSterckMillerSandersWinlaw2010}, a Schur complement based multilevel approach used as a preconditioner for the GMRES iteration~\cite{Virnik2007}, methods based on iterant recombination with minimization in the $\ell_1$- or $\ell_2$-norm~\cite{DeSterckManteuffelMillerSanders2010,DeSterckMillerSanders2011}, and a bootstrap method~\cite{BoltenBrandtBrannickFrommerKahlLivshits2011} which serves as the starting point for the new method presented in this paper. 

The remainder of this paper is organized as follows: In section~\ref{sec:markovchains} and~\ref{sec:amg}, we provide an overview of basic material concerning Markov chains and multigrid methods, respectively. In section \ref{sec:bamg} we give a review of the idea and the various components of the bootstrap algebraic multigrid framework and provide details concerning the use of singular vectors as test vectors. Some special adjustments of the framework, making it suitable for dealing with Markov chains, are described in section \ref{sec:markovbamg}. In section \ref{sec:preconditioned_gmres}, details on how to use the BAMG hierarchy as a preconditioner for a GMRES iteration are given, and the speed of convergence of the resulting method is investigated. Several numerical examples illustrating the performance of the resulting method are given in section \ref{sec:examples}. In section \ref{sec:conclusion}, concluding remarks and an outlook on topics for future research are given. 

\section{Markov chains}\label{sec:markovchains}
A discrete finite Markov chain with states $\{1,\dots,n\}$ and transition probabilities $p_{ij} \geq 0$ between these states can be identified with a matrix $A \in \mathbb{R}^{n \times n}$, the transition matrix, by setting $a_{ij} = p_{ji}$. By this definition the transition matrix is column stochastic, i.e., $a_{ij} \geq 0$ for all $i,j$ and the column sums $\sum_{i=1}^n a_{ij}$ of $A$ are one for all columns $j$. In addition, we will assume that the matrix $A$ is \emph{irreducible}, i.e., there exists no permutation matrix $\Pi$ such that
\begin{equation*}\label{eq:irreducible}
\Pi^TA\Pi = \left[\begin{array}{cc} A_{11} & 0 \\ A_{21} & A_{22}\end{array} \right].
\end{equation*}
A geometric interpretation of irreducibility is that the directed graph induced by $A$ is strongly connected. For an irreducible transition matrix $A$ the Perron-Frobenius theorem (cf.~\cite{BermanPlemmons1994}) implies that the system
\begin{equation}\label{eq:steadystatesystem}
Ax = x
\end{equation}
always has a unique solution $x$ (up to scaling) with all its entries strictly positive.

To determine the steady state vector that solves \eqref{eq:steadystatesystem} by an algebraic multigrid method, we first reformulate the eigenproblem as a linear system 
\begin{equation}\label{eq:steadystatesystemlinear}
Bx = 0, \text{ where } B = I - A.
\end{equation}
Since $1$ is a simple eigenvalue of $A$, the matrix $B$ is singular and $\text{rank}(B) = n-1$.

\section{Algebraic multigrid}\label{sec:amg}
Before describing the bootstrap algebraic multigrid framework in section~\ref{sec:bamg}, we give a very brief review of the basic concepts and individual components of a multigrid algorithm. The first ingredient of any multigrid method is the \emph{smoothing scheme}, which is often a stationary iterative method based on a splitting of the matrix $B$, e.g., weighted Richardson, weighted Jacobi, Gauss-Seidel~\cite{Hackbusch1994} or Kaczmarz relaxation~\cite{Kaczmarz1937}. Any stationary smoothing iteration is based on a matrix splitting $B = F-G$ with $F$ non-singular,
and the smoothing iteration for a system $Bx=b$ reads
\begin{equation} \label{smoothing:eq}
x^{(k+1)} = F^{-1}G x^{(k)} + F^{-1}b =: \mathcal{S}(x^{(k)},b), \enspace k=0,1,\ldots.
\end{equation}  
Its error propagation is given by
\begin{equation}\label{eq:error_propagator_smoothing}
e^{(k+1)} = Se^{(k)} \mbox{ with } S = I - F^{-1}B,
\end{equation}
where $e^{(k)} = x - x^{(k)}$ is the error of the $k$-th iterate $x^{(k)}$. For weighted Jacobi, e.g., we have $F = \omega D$, with $D$ the diagonal of $B$.
The general idea of multigrid methods is based on the observation that smoothing acts 
locally (it combines values of variables which are neighbors in the graph of the matrix) 
and it therefore tends to eliminate certain error components, the ``local'' 
errors, very fast, whereas other, ``global'', components remain almost unchanged. This results 
in slow overall convergence after an initial phase where the norm of the error is reduced substantially. For these local iterative smoothing schemes the ansatz is made that the error components $e$ which cause slow convergence satisfy the inequality
\begin{equation}\label{eq:smooth_error}
\|Be\| \ll \|e\|,
\end{equation}
and they are referred to as \emph{algebraically smooth} error.
Typically, in a multigrid method, a few smoothing iterations are first applied to substantially
reduce the algebraically non-smooth error components. The residual is then restricted to a subspace of smaller dimension, the \emph{coarse grid}, using a restriction operator $Q$, where the 
remaining algebraically smooth error can be treated more efficiently. After computing an approximate representation 
of this error in the subspace from the restricted residual, using a suitably defined  \emph{coarse grid operator} $B_c$, this error representation is interpolated back to the original fine space using an interpolation operator $P$. It is then added as the {\em coarse grid correction} to the current approximation there, and the thus corrected approximation is smoothed again. If the linear system on the coarse grid is solved exactly, the error propagation due to the coarse grid correction process is given by
\begin{equation}\label{eq:coarse_grid_correction}
e^{(k+1)} = e^{(k)} - PB_c^\dagger Qr^{(k)},
\end{equation} where $r^{(k)} = Be^{(k)} = -Bx^{(k)}$ is the residual of~\eqref{eq:steadystatesystemlinear} and $B_c^\dagger$ denotes the Moore-Penrose pseudo inverse of $B_c$, see, e.g.,~\cite{GolubVanLoan1996}.
Including $\nu_1$ pre- and $\nu_2$ postsmoothing iterations, the error propagation matrix of the entire two-grid method is therefore
\begin{equation}\label{eq:error_propagator_twogrid}
E = (I-F^{-1}B)^{\nu_1}(I - PB_c^\dagger QB)(I-F^{-1}B)^{\nu_2}.
\end{equation}
Recursive application of this approach (instead of solving exactly on the coarse grid) gives rise to a multigrid (instead of a two-grid) method.

In this paper, for the sake of simplicity, we only consider $V$-cycles in the solve phase, i.e., we solve the coarse-level equation only once on each level within each multigrid iteration, 
but the method we propose can easily be generalized to incorporate other cycling strategies, cf.~\cite{AxelssonVassilevski1989,TrottenbergOsterleeSchueller2001}. In the following, we always denote the restriction operator by $Q$, the interpolation operator by $P$, and we assume that the coarse grid subspace is given by a subset of the variables of the original fine space. We denote the sets of coarse and fine (and not coarse) variables by $\mathcal{C}$ and $\mathcal{F}$, respectively, with 
$\mathcal{C} \cup \mathcal{F} = \{1,\dots,n\}$ and $\mathcal{C} \cap \mathcal{F} = \emptyset$. This choice of coarse grid variables is often referred to as a $\mathcal{C}/\mathcal{F}$-splitting (cf.~\cite{RugeStueben1986}). As coarse grid operator we choose the Petrov-Galerkin operator $B_c = QBP$. Whenever possible a two-grid notation is used for notational simplicity, with $n_c$ 
standing for the number of coarse variables. In cases where we need to consider all levels of the hierarchy we number them from 1 to $L$, where 1 denotes the finest grid and $L$ the coarsest.

\section{Bootstrap algebraic multigrid (BAMG)}\label{sec:bamg}
In this section, we review the components of the basic BAMG framework and discuss the incorporation of singular vectors as test vectors in the adaptive process. The BAMG framework~\cite{BrandtBrannickKahlLivshits2011,Kahl2009} is a fully adaptive algebraic multigrid method for the solution of general linear systems. It is based on three main components:
\begin{enumerate}
\item compatible relaxation to determine the set of coarse variables $\mathcal{C}$, 
\item least squares based computation of the transfer operators, and 
\item adaptive construction of appropriate test vectors used to build the transfer operators.
\end{enumerate}
This paper focuses on the second and third components of the BAMG process. We refer to~\cite{Brandt2000,JBrannick_2005a,BrannickFalgout2010} for details on compatible relaxation and to \cite{BoltenBrandtBrannickFrommerKahlLivshits2011} for its application to Markov chain problems.

\subsection{Least squares based interpolation}\label{section:LSI}
Given a set of coarse variables, e.g., computed by compatible relaxation, or resulting from geometric coarsening, the next task in the AMG setup phase is to compute a restriction operator, $Q$, and an interpolation operator, $P$, to transfer residuals from fine to coarse grids and to transfer corrections from coarse to fine grids. A careful choice of these operators is crucial to the design of an efficient multigrid method. 

In classical multigrid theory (cf.~\cite{RugeStueben1986}) the \emph{(strong) approximation property} is used to guide the construction of intergrid transfer operators that yield an efficient method. The approximation property can be formulated as the condition
\begin{equation}\label{eq:approx}
\min\limits_{v_c} \|v - Pv_c\|_B^2 \leq \frac{K}{\|B\|_2}\left\langle Bv,Bv\right\rangle \text{ for a small constant $K > 0$,}
\end{equation}
for each $v$ on the fine grid when $B$ is assumed to be positive definite. Its essence is that a vector $v$ needs to be approximated in the range of $P$ with an accuracy that is inversely proportional to $\|Bv\|$. This is motivated by the fact that vectors $v$ with $\|Bv\|\ll \|v\|$ (i.e., algebraically smooth vectors) dominate the error after applying the smoothing iteration and must therefore be reduced via the coarse grid correction. If no a priori information on algebraically smooth vectors is known, a common idea for adaptively constructing suitable transfer operators is to gather information on algebraically smooth vectors in terms of {\em test vectors} and then determine $P$ so that these test vectors lie in or near the range of $P$.
In the context of BAMG, this is done by determining an interpolation operator that approximates a set of given vectors in a least squares sense. Specifically, given a set of test vectors $\mathcal{V} = \{v^{(k)}, k = 1,\dots r$\} such that $\|v^{(k)}\|_2=1$, a weighted least squares problem of the form
\begin{equation}\label{eq:ls_interpolation}
\min \mathcal{L}(P_i) = \sum\limits_{k=1}^r \omega_k \Big(v^{(k)} - \sum\limits_{j \in \mathcal{J}_i} (P_{i})_{j} (Rv^{(k)})_j \Big)^2
\end{equation}
is solved for each row $P_i = (p_{ij})_{j \in \mathcal{J}_i}$ of the prolongation operator $P$ belonging to a fine-level variable $i \in \mathcal{F}$. In (\ref{eq:ls_interpolation}) $R$ denotes the canonical injection which maps every vector onto its coarse grid components and $\mathcal{J}_i$ denotes the index set of (coarse) variables from which the $i$-th variable interpolates, i.e., the non-zero pattern of row $P_i$, cf.~Figure~\ref{fig:interpolation}. The interpolation acts as the identity on the coarse grid variables $\mathcal{C}$ 
such that, when ordering the $\mathcal{F}$ variables first, $P$ has the structure 
\begin{equation}\label{eq:Pstructure}
  P = \left[
    \begin{matrix}
      P_{fc}\\I
    \end{matrix}\right].
\end{equation} 
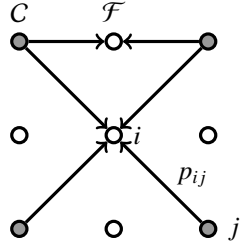
\begin{figure}
\centering
\tikzstyle{Cpoint}=[circle,draw=black,fill=black!40,inner sep=0pt,minimum size=2mm]
\tikzstyle{Fpoint}=[circle,draw=black!100,fill=white,inner sep=0pt,minimum size=2mm]
      \begin{tikzpicture}[very thick, node distance=1cm,auto]
          \node[Cpoint] (11) [label=above:{$\mathcal{C}$}]{};
          \node[Fpoint] (12) [label=above:{$\mathcal{F}$}] [right=of 11] {};
          \node[Cpoint] (13) [right=of 12] {};
          \node[Fpoint] (21) [below=of 11] {};
          \node[Fpoint,label=right:{$i$}] (22) [right=of 21] {};
          \node[Fpoint] (23) [right=of 22] {};
          \node[Cpoint] (31) [below=of 21] {};
          \node[Fpoint] (32) [right=of 31] {};
          \node[Cpoint] (33) [right=of 32,label=right:{$j$}] {};

          \path[<-] (12) edge (11);
          \path[<-] (12) edge (13);
      
          \path[<-] (22) edge (11);
          \path[<-] (22) edge (13);
          \path[<-] (22) edge (31);
          \path[<-] (22) edge node[near end, above,xshift=4pt,yshift=2pt] {$p_{ij}$} (33);
      
        \end{tikzpicture}
\caption{Interpolation naming convention: The $\mathcal{F}$-variable $i$ is interpolated from the  four gray $\mathcal{C}$-variables which form the set $\mathcal{J}_i$.}
\label{fig:interpolation}
\end{figure}

The sets $\mathcal{J}_i$ can be determined using the geometry of the problem or can be chosen from a local neighborhood of $i$ in the graph corresponding to $B$ using a greedy algorithm, see~\cite{BoltenBrandtBrannickFrommerKahlLivshits2011,BrandtBrannickKahlLivshits2011}. When an algebraic approach is used, a threshold $c$ for the number of interpolation variables has to be set to ensure the sparsity of the interpolation operator and, thus, to limit the growth of the operator complexity
\begin{equation} \label{operator_complexity_def:eq}
  o_c = \frac{1}{\mbox{nnz}(B)}\sum_{i = 1}^{L}\mbox{nnz}(B_{i}).
\end{equation}
In most practical settings, a good value of $c$ ranges between 1 and 4. For details on the choice of $\mathcal{J}_i$, we refer to~\cite{BoltenBrandtBrannickFrommerKahlLivshits2011,BrandtBrannickKahlLivshits2011}. The weights, $\omega_k$, are chosen such that the algebraically smoothest vectors have the largest weights in the minimization process. A suitable choice, which we use exclusively in our numerical experiments, is given by
\begin{equation}\label{eq:ls_weights}
\omega_k = \frac{1}{\|Bv\|^2}.
\end{equation}

In~\cite{BrezinaManteuffelMcCormichRugeSanders2010}, a heuristic argument was proposed suggesting the use of (approximated) singular vectors as test vectors in the nonsymmetric case. The argument follows from the singular value decomposition $B = U\Sigma V^T$.  Let the matrix $W = VU^T$ and define the symmetric positive semidefinite matrices $WB = \sqrt{B^TB}$ and $BW = \sqrt{BB^T}$. Then the original nonsymmetric system $Bx=0$ can be reformulated in two ways as an equivalent symmetric system using $WB$ or $BW$ as the system matrix. The eigenvectors corresponding to the minimal eigenvalues of $WB$ are the \emph{right} singular vectors corresponding to the minimal singular values of $B$ and those for $BW$ are the \emph{left} singular vectors corresponding to minimal singular values of $B$. Because $WB$ and $BW$ are symmetric positive semi-definite, this fact can be used to derive an approximation property for the original problem involving $B$, assuming that $P$ and $Q$ are based on singular vectors. 

Determining the restriction operator $Q$ for $B$ is equivalent to determining the transpose of the interpolation operator for $B^T$. This
observation naturally leads to a least squares problem similar to \eqref{eq:ls_interpolation} for each column of $Q$, where the weights are now given by $1/\|B^Tv\|^2$. This construction, in turn, requires an additional set $\mathcal{U} = \{u^{(k)}, k = 1,\dots,r\}$, of left test vectors.  An efficient multilevel algorithm for constructing the sets $\mathcal{U}$ and $\mathcal{V}$ will now be described.

\subsection{Bootstrap construction of test vectors} \label{bootstrap:subsec}
In this section, a BAMG cycle for computing approximations of singular vectors with small singular values is developed. The algorithm is based on the approach from~\cite{BrandtBrannickKahlLivshits2011} for approximating small eigenvalues and the corresponding eigenvectors for symmetric operators (where $Q=P^T$, $B_c = P^TBP$). In this approach, a preliminary multigrid hierarchy is constructed and (generalized) eigenvectors
\begin{equation}\label{eq:original_bamg_eigenproblem}
P^TBPv = \lambda P^TPv
\end{equation}
of the operator on the coarsest grid are determined and interpolated to the finest grid in a suitable way. They are then used as new test vectors for another setup cycle. This approach can be extended to singular vectors by noting that the problem of computing singular vectors is equivalent to solving a symmetric eigenproblem of twice the size, cf.~\cite{GolubVanLoan1996}.
This has also been observed in~\cite{DeSterck2011} in the context of adaptive algebraic 
multigrid methods for computing singular triplets. In the following, we describe the details of this construction.\\

To determine suitable approximations to the singular vectors, we consider the equations
\begin{equation}\label{eq:singular_vectors}
\left.
\begin{array}{rcl}
Bv_i &=& \sigma u_i \\
B^Tu_i &=& \sigma v_i
\end{array} \right\} \enspace \text{ for }i = 1,\dots,n,
\end{equation}
characterizing the singular vectors of $B$. These equations can be combined into one system of the form
\begin{equation}\label{eq:singular_vectors_2n}
\left[\begin{array}{cc} 0 & B \\ B^T & 0 \end{array}\right] \left[\begin{array}{cc} U & U \\ V & -V \end{array}\right] = \left[\begin{array}{cc} U & U \\ V & -V \end{array}\right] \left[\begin{array}{cc} \Sigma & 0 \\ 0 & -\Sigma
\end{array}\right],
\end{equation} where $U=\left[ \ u_1 \mid \dots \mid u_n \ \right],
V=\left[ \ v_1 \mid \dots \mid v_n \ \right]\text{ and }
\Sigma= \diag(\sigma_1,\dots, \sigma_n).$
Thus the eigenvalues of the symmetric matrix 
$$\widehat{B}=\left[\begin{array}{cc} 0 & B \\ B^T & 0 \end{array}\right]$$
are the singular values of $B$ and the corresponding eigenvectors contain the singular vectors of $B$. Hence, we can apply the original bootstrap process to the symmetric matrix $\widehat{B}$. To this end define
\begin{equation} \label{Phatdef:eq}
\widehat{P} = \left[\begin{array}{cc} Q^T & 0 \\ 0 & P \end{array}\right]
\end{equation} 
and observe that in the variational sense an eigenpair $(\widehat{x},\widehat{\lambda})$ for $\widehat{B}$, satisfying $\widehat{B}\widehat{x} = \widehat{\lambda}\widehat{x}$, corresponds to an eigenpair $(x_c,\lambda_c)$ of the generalized eigenproblem $\widehat{P}^T\widehat{B}\widehat{P}x_c = \lambda_c \widehat{P}^T\widehat{P}x_c$ on the coarse level,
cf.~\cite{BrandtBrannickKahlLivshits2011,Kahl2009}. With the choice \eqref{Phatdef:eq} of $\widehat{P}$, we obtain the generalized (symmetric) eigenvalue problem on the coarse grid, 
\begin{equation}\label{eq:singular_vectors_2n_coarse}
\left[\begin{array}{cc} 0 & QBP \\ (QBP)^T \hspace{-0.13cm} & 0 \end{array}\right] \hspace{-0.13cm}
\left[\begin{array}{cc} U_c & U_c \\ V_c & \hspace{-0.13cm} -V_c \end{array}\right] = \left[\begin{array}{cc} QQ^T & 0 \\ 0 & \hspace{-0.13cm} P^TP \end{array}\right] \hspace{-0.13cm} \left[\begin{array}{cc} U_c & U_c \\ V_c & \hspace{-0.13cm} -V_c \end{array}\right] \hspace{-0.13cm} \left[\begin{array}{cc} \Sigma_c & 0 \\ 0 & \hspace{-0.13cm}-\Sigma_c \end{array}\hspace*{-0.1cm}\right], 
\end{equation}
a generalization of \eqref{eq:singular_vectors_2n}.
Alternatively one can also derive \eqref{eq:singular_vectors_2n_coarse} using a Petrov-Galerkin approach as follows: Assume that a singular value $\sigma$ of $B$ is known and some lower-dimensional subspaces $\mathcal{K}_1,\mathcal{K}_2$ are given. We search for vectors $u_c \in \mathcal{K}_1, v_c \in \mathcal{K}_2$ which approximately satisfy
\begin{equation}\label{eq:singular_vectors_galerkin}
\begin{array}{rcl}
Bv_c &=& \sigma u_c,\\
B^Tu_c &=& \sigma v_c.
\end{array}
\end{equation}
To identify the approximations $u_c, v_c$ we demand that the residuals $Bv_c-\sigma u_c$ and $B^Tu_c-\sigma v_c$ be orthogonal to some subspaces $\mathcal{L}_1,\mathcal{L}_2$ with $\dim\mathcal{L}_i = \dim\mathcal{K}_i$. If we set $\mathcal{K}_1 = \mathcal{L}_2 = \range(Q^T)$, the column span of the matrix $Q^T$, and $ \mathcal{K}_2 = \mathcal{L}_1 = \range(P)$, \eqref{eq:singular_vectors_galerkin} can be rewritten as
\begin{equation}\label{eq:singular_vectors_coarse}
\begin{array}{rcl}
QBPv_c - \sigma QQ^Tu_c &=& 0,\\
P^TB^TQ^Tu_c - \sigma P^TPv_c &=& 0,
\end{array}
\end{equation}
which is equivalent to \eqref{eq:singular_vectors_2n_coarse}.
Recursive application of this approach leads to
\begin{equation}\label{eq:singular_vectors_coarsest}
\begin{array}{rcl}
B_l v_l - \sigma M_l u_l &=& 0,\\
B_l^Tu_l - \sigma N_l v_l &=& 0,
\end{array}
\end{equation}
where $M_l = Q_l \cdots Q_1 Q_1^T \cdots Q_l^T$ and $ N_l = P_l^T\cdots P_1^TP_1\cdots P_l$ are the accumulated interpolation and restriction operators, respectively.
On the coarsest level, all possible solutions to \eqref{eq:singular_vectors_coarsest}
are obtained via the generalized eigenvalue problem
\begin{equation}\label{eq:generalized_eigenvalue}
\left[\begin{array}{cc} 0 & B_L \\ B_L^T & 0 \end{array}\right] \left[\begin{array}{cc} U_L & U_L \\ V_L & -V_L \end{array}\right] = \left[\begin{array}{cc} M_L & 0 \\ 0 & N_L \end{array}\right] \left[\begin{array}{cc} U_L & U_L \\ V_L & -V_L \end{array}\right] \left[\begin{array}{cc} \Sigma_L & 0 \\ 0 & -\Sigma_L \end{array}\right]. 
\end{equation} 
As long as all $P_l$ and $Q_l$ have full rank, which is always the case in our approach, 
cf.~\eqref{eq:Pstructure}, $M_L$ and $N_L$ are symmetric positive definite and the existence of solutions to \eqref{eq:generalized_eigenvalue} is guaranteed, cf.~\cite{DeSterck2011}.

When the solutions of \eqref{eq:generalized_eigenvalue} are computed, they must be transferred 
back to the finest grid via the interpolation operators, and the quality of approximation is 
improved by applying the smoothing iteration. To achieve better results, smoothing is not done 
with respect to homogeneous equations, but with respect to the inhomogeneous equations
 $B_l v_l = \sigma M_l u_l$ (smoothing for $v_l$) and $B_l^T u_l = \sigma N_l v_l$ (for $u_l$) resulting from \eqref{eq:singular_vectors_coarsest}. 
After a smoothing step for each of $u_l$ and $v_l$, the subsequent smoothing step uses the smoothed values for $v_l$ and $u_l$ in the right hand sides together with an updated value for the approximate singular value 
\begin{equation}\label{eq:singular_value}
\sigma = \frac{u_l^T B_l v_l}{(u_l^TM_lu_l)^{1/2}(v_l^TN_lv_l)^{1/2}}.
\end{equation}
This update for $\sigma$ can be regarded as a generalized Rayleigh quotient. It yields the exact singular value if $u_l$ and $v_l$ are exact generalized singular vectors of $B_l$ in the sense of \eqref{eq:singular_vectors_coarsest}.

To start the above procedure, preliminary interpolation and restriction operators are needed. 
To get some first information, a few smoothing iterations are applied to the two systems
$Bx=0 \text{ and }  B^Tx=0$
to $r$ independent, random initial vectors for each of the matrices, $B$ and $B^T$. The resulting vectors are then smooth enough to allow for the construction of a first multigrid hierarchy when used as test vectors. This first hierarchy can then be used to produce better approximations to the singular vectors to the $r$ smallest singular vectors of $B$, which can then in turn be used as new test vectors to construct an even better multigrid hierarchy. In this manner, the method continuously improves itself, which explains why it is referred to as a bootstrap method. We summarize the first cycle of the bootstrap setup in Algorithm~\ref{alg:bamg_setup}, where we
use the notation $\mathcal{S}_l$ and $\mathcal{S}_l^T$ to denote the smoothing iteration 
\eqref{smoothing:eq} for $B_l$ and $B_l^T$, respectively, on level $l$. 
Subsequent cycles proceed in the same manner except that the first smoothing in lines~\ref{smooth1:line}-\ref{smooth2:line} then uses the information which is already available, i.e.\ it proceeds as in lines~\ref{smooth3:line}-\ref{smooth4:line}.

\begin{algorithm}[h]
\DontPrintSemicolon
\SetKwFor{For}{for}{}{end}
\SetArgSty{}
\KwIn{$B_l$ $(B_0 = B)$, $M_l$ $(M_0 = I)$, $N_l$ $(N_0 = I)$, test vectors $\mathcal{U}_l, \mathcal{V}_l)$}
\KwOut{Triplets $(\Sigma_l, \mathcal{U}_l, \mathcal{V}_l)$ of approximate singular vectors and values of $B$}
\eIf{$l = L$ (coarsest grid is reached)}{
	  Determine the $r$ smallest generalized singular triplets $(\Sigma_l, \mathcal{U}_l, \mathcal{V}_l)$ of \eqref{eq:generalized_eigenvalue}.
}{
		\For{no.\ of smoothing steps and $k=1,\ldots,r$}{
		$\widetilde{v}_l^{(k)} = \mathcal{S}_l\big(v_l^{(k)},0\big)$. \enspace 
		$\widetilde{u}_l^{(k)} = \mathcal{S}^T_l\big(u_l^{(k)},0\big)$. \label{smooth1:line}  \;
                ${v}_l^{(k)} = \widetilde{v}_l^{(k)}, \enspace {u}_l^{(k)} = \widetilde{u}_l^{(k)}$. \label{smooth2:line}\;
                }
		\For{$m = 1,\dots,\mu$}{
			Compute the rows of $P_l$ and columns of $Q_l$ via \eqref{eq:ls_interpolation}. \;
			Compute $B_{l+1} = Q_lB_lP_l$.\;
		        Compute $M_{l+1} = Q_lM_lQ_l^T$.\;
		        Compute $N_{l+1} = P_l^TN_lP_l$.\;
			Set $\mathcal{U}_{l+1} = \{R_lu_l^{(k)} : k = 1,\dots,r\}$.\;
			Set $\mathcal{V}_{l+1} = \{R_lv_l^{(k)} : k = 1,\dots,r\}$.\;
			$(\Sigma_{l+1}, \mathcal{U}_{l+1}, \mathcal{V}_{l+1})$ = bamg\_mle$(B_{l+1}, M_{l+1}, N_{l+1}, \mathcal{U}_{l+1}, \mathcal{V}_{l+1}))$.\;
			Set $(\Sigma_{l+1}, \mathcal{U}_{l+1}, \mathcal{V}_{l+1}) = \{(\sigma_{l+1}^{(k)}, Q_l^Tu_{l+1}^{(k)}, P_lv_{l+1}^{(k)}) : k = 1,\dots,k_v\}$.\;
			\For{no.\ of smoothing steps and $k=1,\ldots,r$}{
			    $\widetilde{v}_l^{(k)} = \mathcal{S}_l\big(v_l^{(k)},\sigma_l^{(k)} M_l u_l^{(k)}\big)$. \enspace 
			    $\widetilde{u}_l^{(k)} = \mathcal{S}^T_l\big(u_l^{(k)},\sigma_l^{(k)} N_l v_l^{(k)}\big)$. \label{smooth3:line}\;
                            $v_l^{(k)} = \widetilde{v}_l^{(k)}, \enspace u_l^{(k)} = \widetilde{u}_l^{(k)}, \enspace \sigma_l^{(k)} = \frac{(u_l^{(k)})^T B_l v_l^{(k)}}{((u_l^{(k)})^TM_lu_l^{(k)})^{1/2}((v_l^{(k)}l)^TN_lv_l^{(k)})^{1/2}}$. \label{smooth4:line} \;
                        } 
		}
}

\caption{bamg\_mle (BAMG setup, first cycle)\label{alg:bamg_setup}}
\end{algorithm}

\section{Application of BAMG to Markov chains}\label{sec:markovbamg}
We now go into detail concerning the application of BAMG to computing the steady state vector of a discrete Markov chain. We are interested in solving a linear system with the matrix
\begin{equation}\label{eq:BIA}
B = I-A,
\end{equation}
where $A$ is the transition matrix of a Markov chain and therefore column stochastic, implying
that all column sums of $B$ equal zero, i.e., $\mathbf{1}^T B = 0$ with $\one =
(1,\ldots,1)^T$. Since we also assume that $A$ and thus $B$ is irreducible, the left and right nullspaces of $B$ both have dimension one. Nullspaces are identical to the spaces spanned by the singular vectors for the singular value 0. In other words: we know the left singular vector to the smallest singular value, $\sigma_n = 0$, which is $\one$,
and we want to compute the unknown corresponding right singular vector $x= v_n.$ The bootstrap construction of test vectors of section~\ref{bootstrap:subsec} will compute an approximation to
$v_n$ together with approximations to further left and right singular vectors belonging to small
singular values. 

It is reasonable to try to take advantage of the additional information that $\one$ is a left singular vector for the singular value 0.  
We do so by preserving the property $\mathbf{1}^TB = 0$ on the coarser levels. On all levels we then know the left singular vector for the smallest singular value exactly.
To achieve this, we use the exact singular vector $\mathbf{1}$ as a test vector within the least squares process (\ref{eq:ls_interpolation}) to define $Q$. For $u^{(1)} = \mathbf{1}$ this results in the weight $\omega_1 = \infty$ with respect to $B^T$ according to \eqref{eq:ls_weights}, 
\begin{equation}\label{eq:weight_infinity}
\omega_1 = \frac{1}{\|B^T\mathbf{1}\|^2} = \frac{1}{0} = \infty.
\end{equation}
Therefore, the least squares functional $\mathcal{L}$ will attain the value $\infty$ as long as the vector $\mathbf{1}$ is not interpolated exactly. As a modification to the bootstrap process for Markov chains we thus require that $\mathbf{1}$ be interpolated exactly and set $0\cdot \infty = 0$, so that \eqref{eq:ls_interpolation} turns into a least squares problem for the other test vectors $u^{(2)},\dots,u^{(r)}$, subject to the restriction $Q_i^T\mathbf{1} = \mathbf{1}$ for all columns $i$ of $Q$. Restricted least squares problems of this kind can easily be solved by known techniques or transformed into standard least squares problems, see~\cite{GolubVanLoan1996}. Overall, the restrictions for the individual columns of $Q$ imply that the computed restriction operator will satisfy $\mathbf{1}^TQ = \mathbf{1}^T$, leading to $\one^TB_c = \mathbf{1}^TQBP = \mathbf{1}^T BP = 0$. In our numerical experiments we noticed that preserving the constant in this manner improves the quality of the coarse grid operators when compared to the simple averaging (the non-zero structure of the restriction operator $Q$ is chosen as the transpose of the non-zero structure of $P$ with all entries in one column set to $\frac{1}{c}$, where $c$ is the number of non-zero entries in this column) from \cite{BoltenBrandtBrannickFrommerKahlLivshits2011,DeSterckManteuffelMcCormickNguyenRuge2008,DeSterckManteuffelMcCormickRugeMillerPearsonSanders2010,DeSterckMillerSandersWinlaw2010}, another 
attempt to preserve structural properties for the coarse grid operators.

\section{BAMG-preconditioned GMRES}\label{sec:preconditioned_gmres}
Each setup cycle, as described in section \ref{sec:bamg}, involves recomputing the operators $P$, 
$Q$, and then $B_c = QAP$, and is therefore computationally quite costly.  As demonstrated  
in~\cite{BoltenBrandtBrannickFrommerKahlLivshits2011} and then also later explored in~\cite{DeSterck2011}, the use of a second, {\em additive} phase---where one iterates with a fixed multigrid hierarchy to improve the quality of the 
approximation to the state vector obtained so far---can result in substantial speed up.
This idea is motivated by the fact 
that the multigrid hierarchy developed by the setup cycles
approximates a rich subspace spanned by singular 
vectors of $B$ with small singular values (in addition 
to the one with singular value $0$, the state vector) which makes the resulting 
hierarchy well suited for forming an effective preconditioner
for GMRES.  Specifically, let $E_0 = I - \widetilde{C}B$ denote 
the two-grid error propagation 
operator \eqref{eq:error_propagator_twogrid}. 
Then, given the approximation $x^{(0)}$ to the steady state vector $x$, 
computed in the setup phase, 
and an arbitrary initial guess $e^{(0)}$ for the residual equation 
\begin{equation} \label{residual:eq}
 Be = -Bx^{(0)},
\end{equation}
one application of the multigrid V-cycle preconditioner yields a new iterate
\begin{equation}\label{eq:iterate_preconditioner}
e^{(1)} = (I-\widetilde{C}B)e^{(0)} - \widetilde{C}Bx^{(0)}.
\end{equation}
Thus, if we start with initial guess $e^{(0)} = 0$, one multigrid iteration corresponds to a multiplication with the implicitly defined preconditioner $C = \widetilde{C}$, allowing us to 
use GMRES on the preconditioned residual equation
\begin{equation} \label{prec_residual:eq}
CBx = -CBx^{(0)}.
\end{equation} 

For singular linear systems, convergence of the GMRES iteration can in general not be guaranteed, 
even when the linear system is consistent. 
The following theorem on the convergence of GMRES for singular systems summarizes results from~\cite{IpsenMeyer98}. Recall that $B$ is said to have index 1 if $\range(B) \cap \nullsp(B) = 0$.

\begin{theorem} \label{sing_GMRES_conv:thm} Assume that the matrix $B$ is singular and has index 1.
Then the GMRES iterates for the singular system $Bx = b$ with $b \in \range(B)$ converge
towards the solution $x = x^{(0)} + B^D(b-Bx^{(0)})$, where $x^{(0)}$ is the 
starting vector and $B^D$
denotes the Drazin inverse of $B$.
\end{theorem}

For the matrix $B = I-A$ with $A$ the transition matrix of an irreducible Markov chain, we have $\nullsp(B) = \left\langle x \right\rangle$ and $\range(B) = \left\langle \mathbf{1} \right\rangle^\perp$ where $x$ is the steady state vector of the Markov chain. Because all components of $x$ are strictly positive, $x$ can not be orthogonal to the constant vector $\mathbf{1}$ and therefore $B$ has index 1. Since the residual equation 
\eqref{residual:eq} is consistent, Theorem~\ref{sing_GMRES_conv:thm} gurantees the convergence of the GMRES iteration.

It is not immediately clear that the preconditioned matrix $CB$ also has index 1, since $CB$ in general will not have column sums all equal to zero, so that we will no longer have that $\range(CB)  = \left\langle \mathbf{1} \right\rangle^\perp$. 

However in all our numerical experiments from section~\ref{sec:examples} we never experienced a breakdown of GMRES, indicating that it is highly unlikely in practice that $CB$ has index $> 1$. We proceed by formulating results which, to a certain extent, quantify the acceleration of convergence due to the preconditioning.
Theorem~\ref{starke:thm} below recalls a result from \cite{EiermannErnst2001} and \cite{Starke97} on the convergence of GMRES in the non-singular case. It uses the field of values $\mathcal{F}(B)$ of $B$ defined as
\[
   \mathcal{F}(B) = \{ \langle Bx,x\rangle: \langle x, x\rangle = 1, x \in \mathC^n \}.
\]  
$\mathcal{F}(B)$ is a compact set containing all the eigenvalues of $B$, see \cite{HornJohnson:topics}. We denote by $\nu(\mathcal{F}(B))$ the distance of the field of values to the origin, i.e.\
\[
\nu(\mathcal{F}(B))   = \min_{z \in \mathcal{F}(B)} \|z\|.
\]

\begin{theorem} \label{starke:thm} The residuals $r_k = b-Bx_k$ of the GMRES iterates $x_k$ satisfy
\begin{equation} \label{starke_est:eq}
\frac{\|r_k\|}{\|r_0\|} \leq \left( 1-\nu(\mathcal{F}(B)) \nu(\mathcal{F}(B^{-1}))\right)^{k/2}. 
\end{equation}
\end{theorem}

It is easy to see that $\nu(\mathcal{F}(B)) > 0 $ if and only if $\nu(\mathcal{F}(B^{-1})) > 0 $,
and that $\nu(\mathcal{F}(B)) \cdot \nu(\mathcal{F}(B^{-1}))<1.$
Thus the theorem predicts progress in the GMRES iteration in every step provided  
$\nu(\mathcal{F}(B)) > 0 $, and the larger $\nu(\mathcal{F}(B)) \cdot \nu(\mathcal{F}(B^{-1}))$ is, the faster the predicted progress is. It has been shown in 
\cite{EiermannErnst2001,LiesenTichy12} that \eqref{starke_est:eq} is an improvement over the estimate from \cite{Elman82} (see also \cite{Saad1996,SaadSchultz1986})
\begin{equation} \label{elman_bound:eq}
 \frac{\|r_k\|}{\|r_0\|} \leq \left(1-(\nu(\mathcal{F}(B))/\|B\|)^2 \right)^{k/2}
\end{equation}
in the sense that $1-\nu(\mathcal{F}(B)) \nu(\mathcal{F}(B^{-1}) \leq 1-(\nu(\mathcal{F}(B))/\|B\|)^2$.

In order to, at least partly, explain the fast convergence of the BAMG-precon\-di\-tio\-ned GMRES iteration we now relate the GMRES iteration for the singular system to one for a
non-singular system of smaller dimension via the following theorem based on \cite[Thm.2.9]{HayamiSugihara2011}.

\begin{theorem} \label{sing_gmres2:thm} Assuming the singular matrix $B$ has index 1 and $\dim(\range(B)) = m < n$. Let $\Pi \in \mathR^{n \times m}$ have orthonormal columns which span $\range(B)$ and define
\[
\widehat{B} = \Pi^TB\Pi \in \mathR^{m \times m}.
\]
Assume that $b \in  \range(B)$ and consider the GMRES iterates $x_k$ for the system $Bx = b$ with starting vector $x_0.$ Decompose $x_0 = x_0^1 + x_0^2$ with 
$x_0^1 \in \range(B), x_0^2 \in \nullsp(B)$ and consider the GMRES iterates for the system $\widehat{B}\widehat{x} = \widehat{b}$ with $\widehat{b} = \Pi^Tb$ and $\widehat{x}_0 = \Pi^Tx_0^1$.
Then we have
\begin{equation} \label{sing_nonsing_rel:eq}
x_k = x_0^2 + \Pi\widehat{x}_k \mbox{ and } r_k = b-Bx_k = \Pi(\widehat{B}(\widehat{b}-\widehat{x}_k) = \Pi\widehat{r}_k.
\end{equation}
In particular, if $\nu(\fov{\widehat{B}}) > 0 $ we have
\[
\frac{\|r_k\|}{\|r_0\|} \leq \left(1-\nu(\mathcal{F}(\widehat{B})) \nu(\mathcal{F}(\widehat{B}^{-1})) \right)^{k/2}.
\]
\end{theorem}  
\begin{proof}
The equality \eqref{sing_nonsing_rel:eq} was established in \cite{HayamiSugihara2011}. 
The subsequent estimate follows from Theorem~\ref{starke:thm}.
\end{proof}

Using this theorem we can substantiate the experimentally observed fast 
convergence of the BAMG preconditioned GMRES iteration as compared to plain GMRES by comparing the fields of values of the respective projected matrices $\widehat{CB}$ and $\widehat{B}$, respectively.  
Note that since $\langle \widehat{B}y,y \rangle = \langle B\Pi y, \Pi y \rangle$ and $\|\Pi y\| = \| y\|$ we have
\[
\fov{\widehat{B}} = \{ \langle \widehat{B}y,y \rangle: \|y \| = 1, y \in \mathC^m \} =  
\{ \langle Bz,z \rangle: \|z \| = 1, z \in \mathC^n, z = \Pi y\}
\subseteq \fov{B}.
\]

\begin{figure}
\centerline{Uniform two-dimensional network ($n = 1089$):}
\centerline{
\includegraphics[width = 0.3\textwidth]{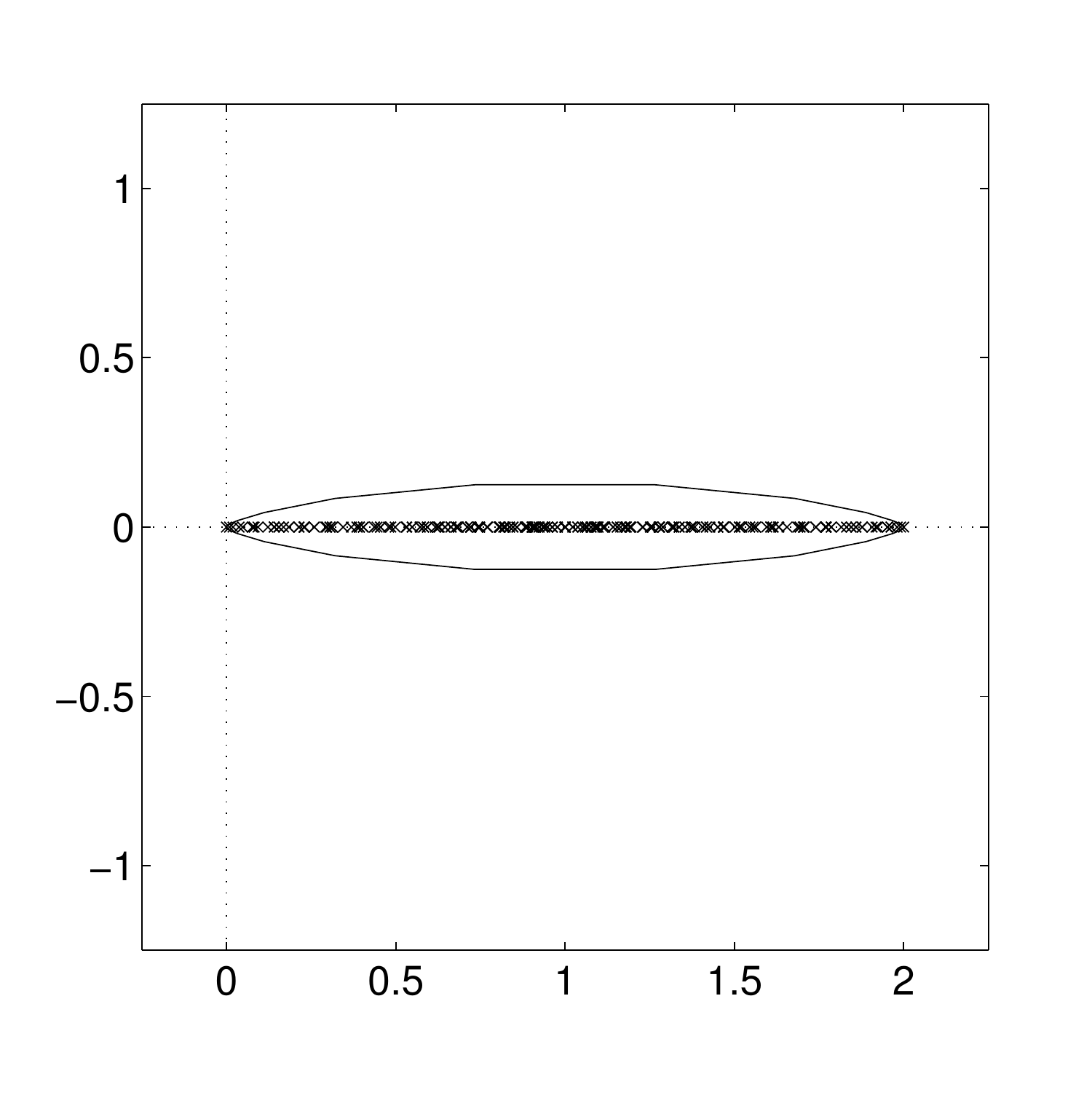}
\includegraphics[width = 0.3\textwidth]{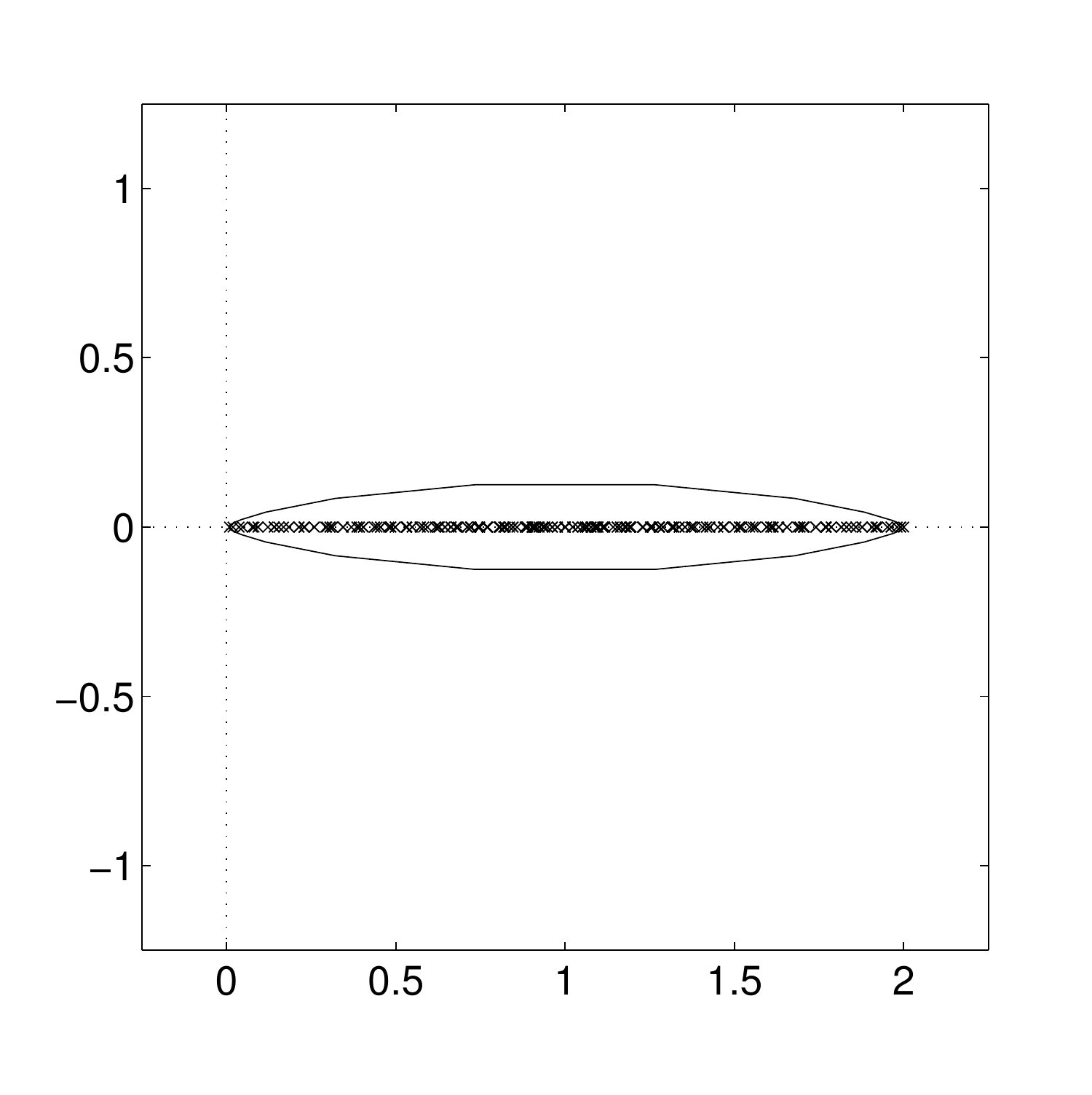}
\includegraphics[width = 0.3\textwidth]{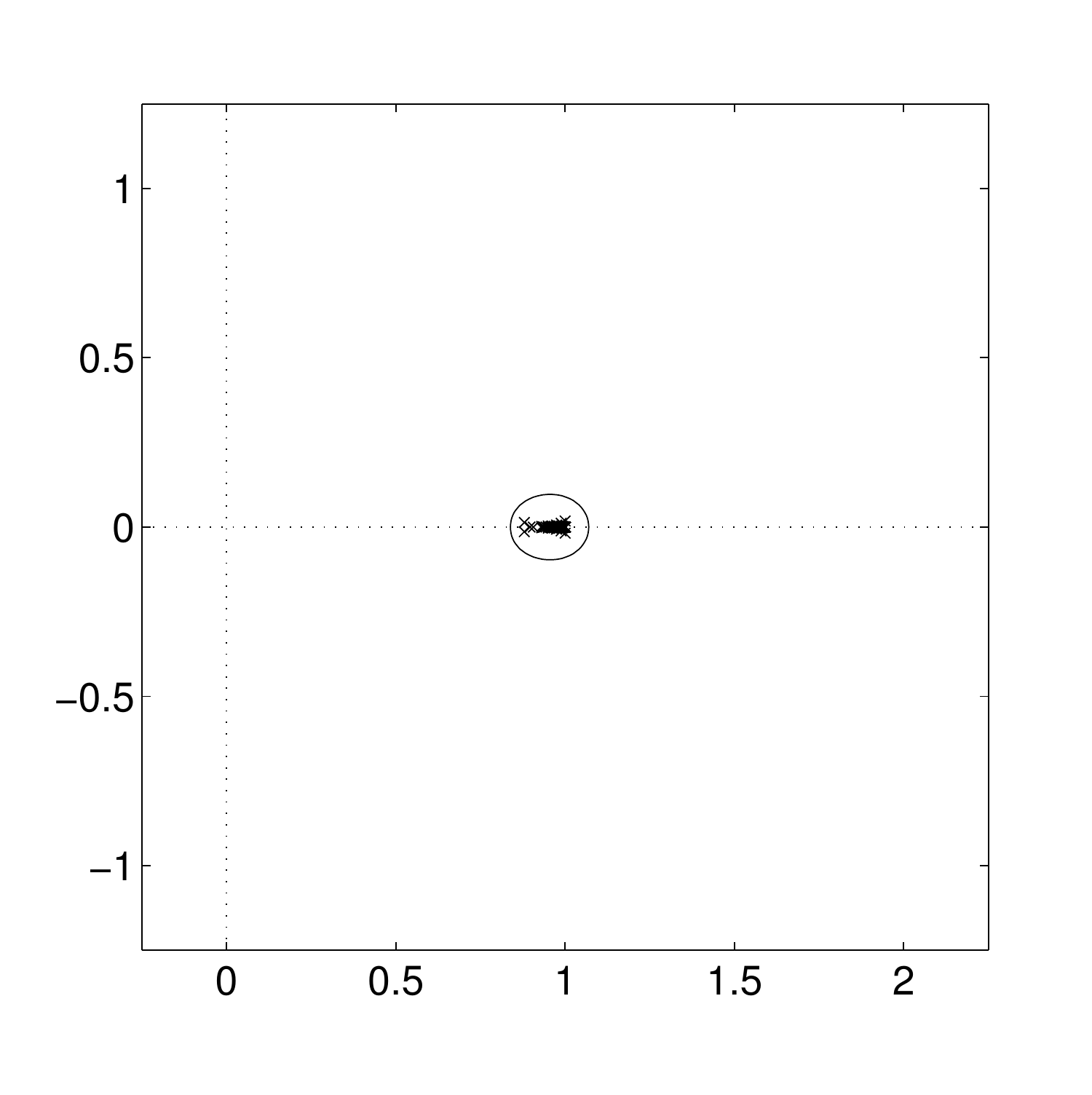}
} 
\centerline{Tandem queuing network ($n = 1089$):}
\centerline{
\includegraphics[width = 0.3\textwidth]{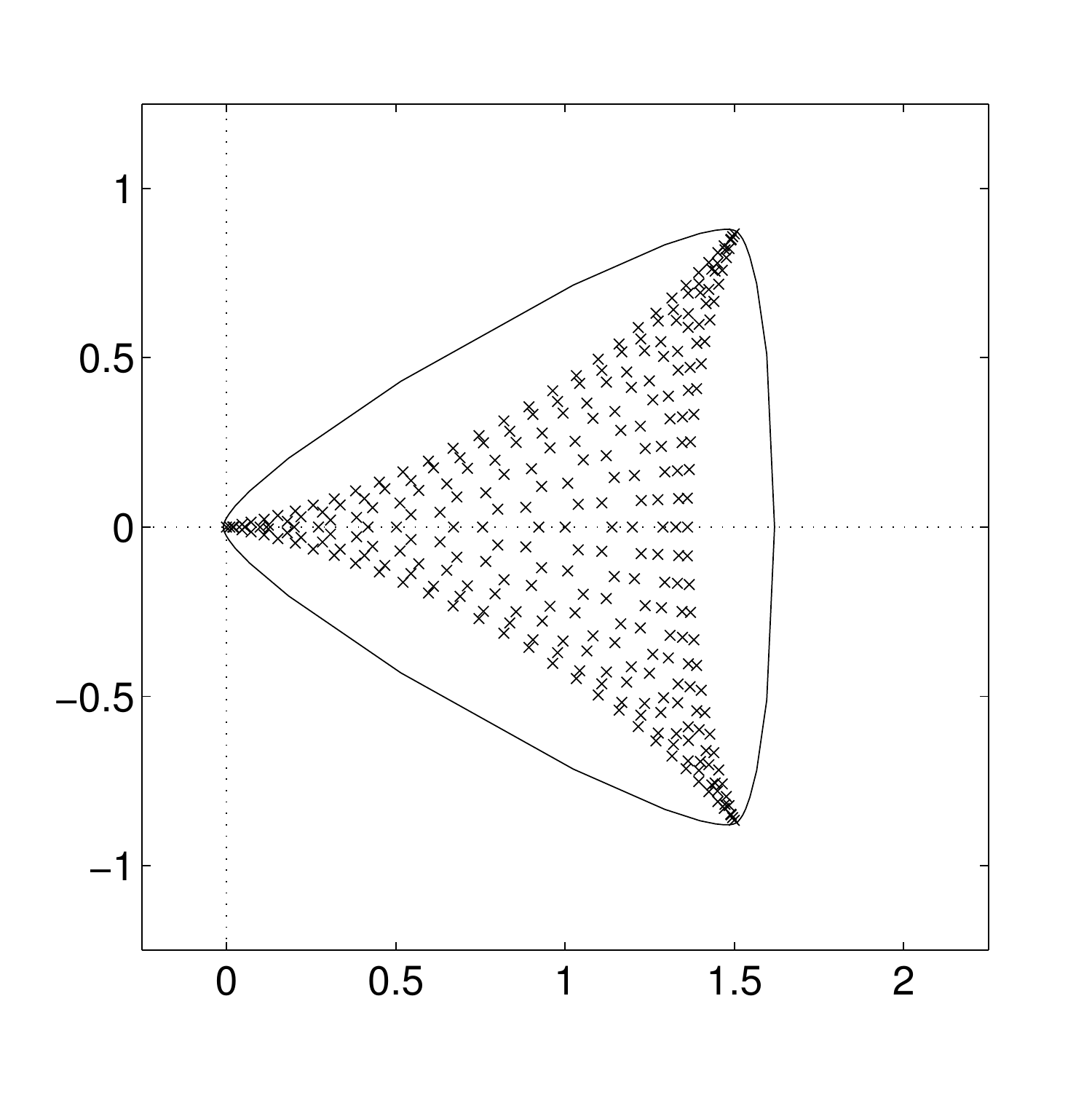}
\includegraphics[width = 0.3\textwidth]{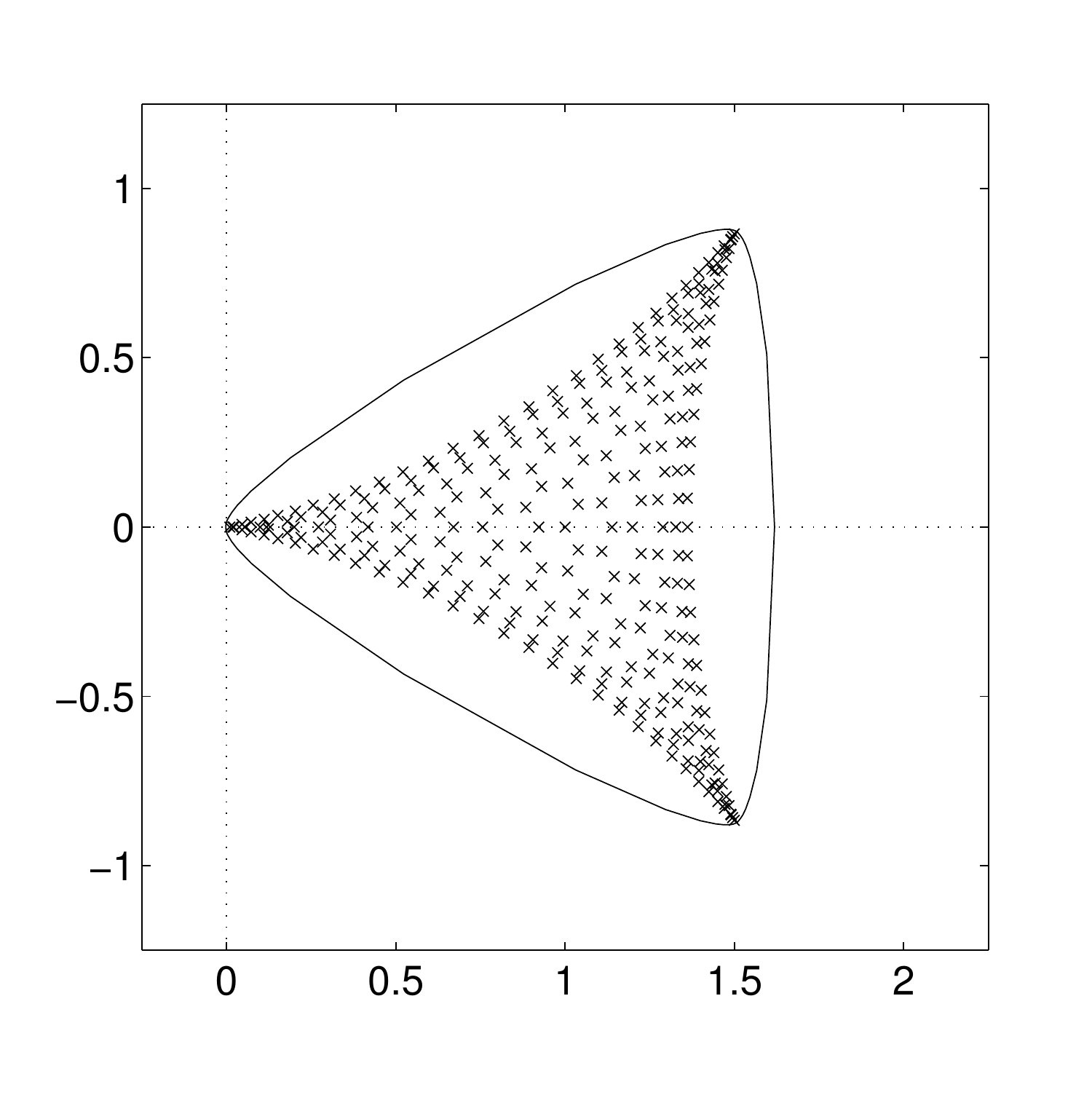}
\includegraphics[width = 0.3\textwidth]{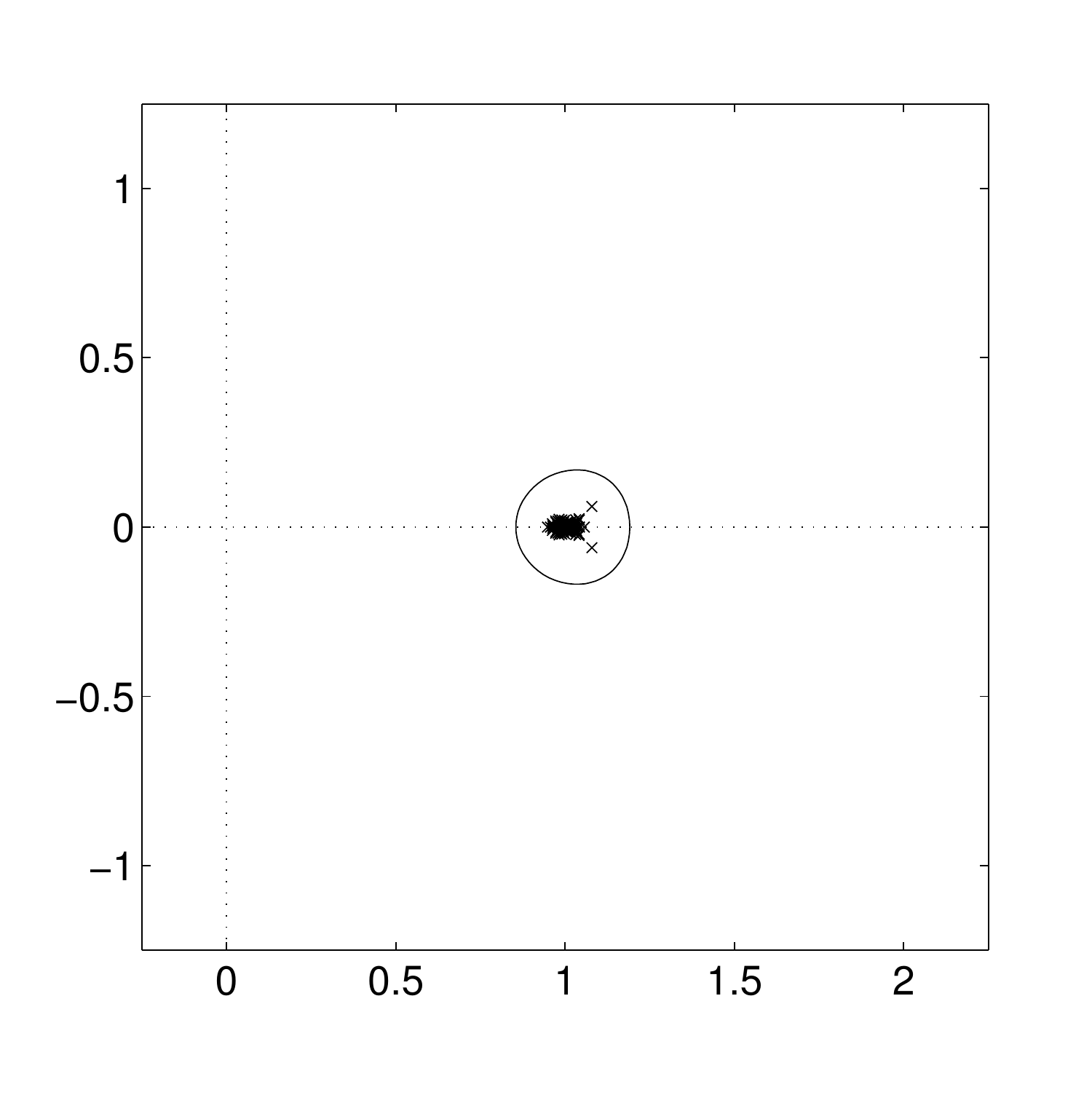}
}
\centerline{Petri net ($n = 506$):}
\centerline{
\includegraphics[width = 0.3\textwidth]{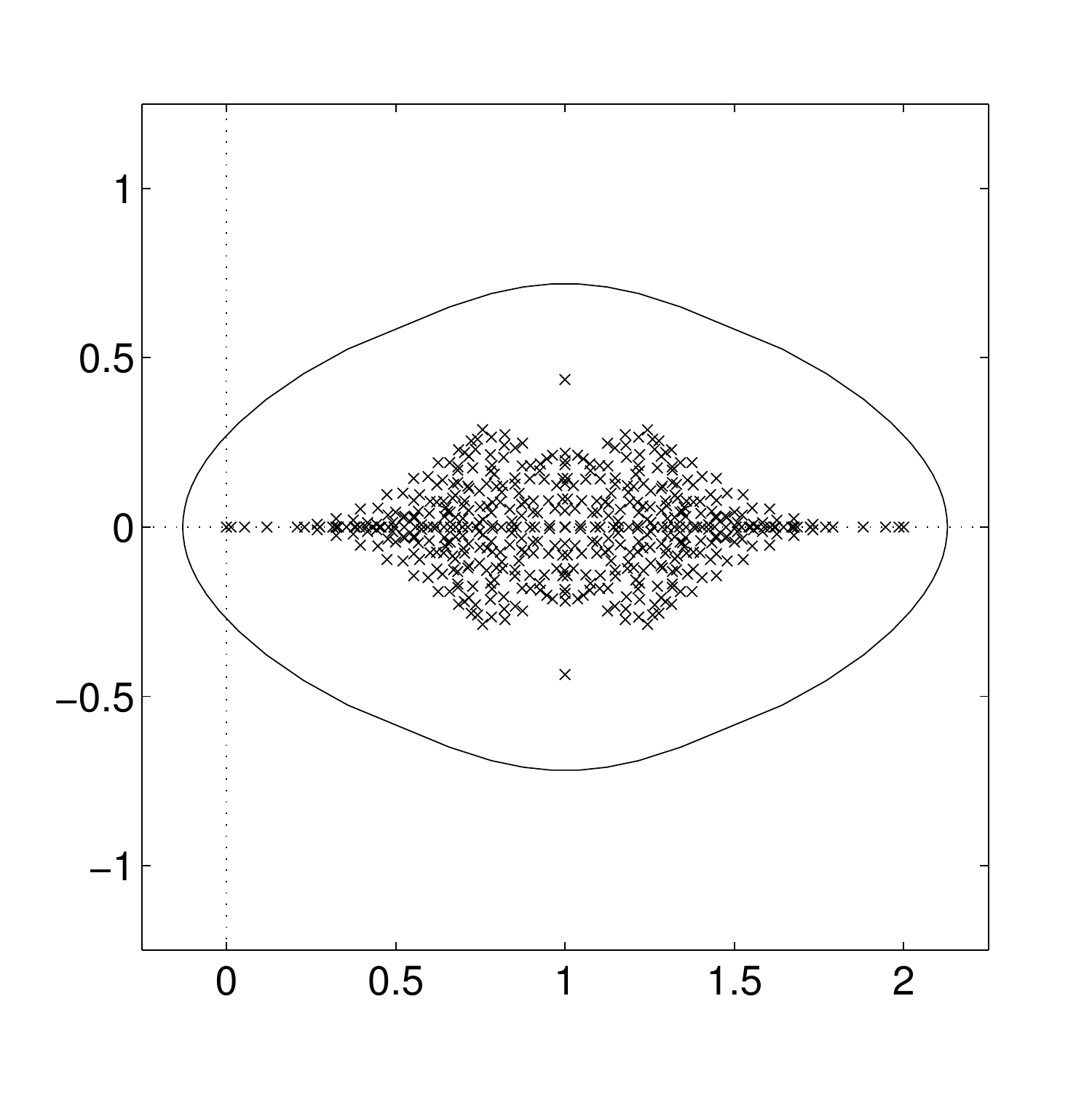}
\includegraphics[width = 0.3\textwidth]{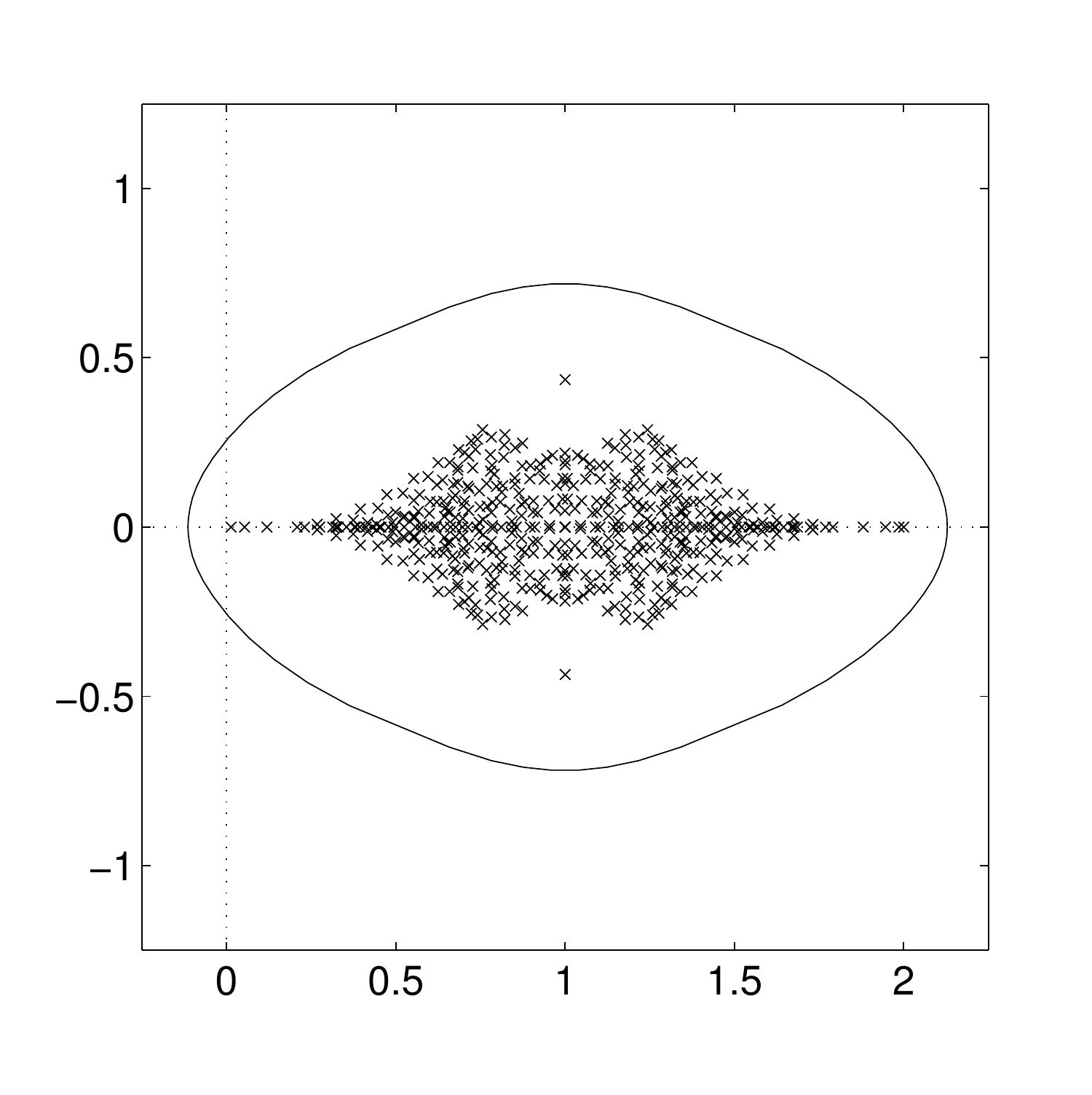}
\includegraphics[width = 0.3\textwidth]{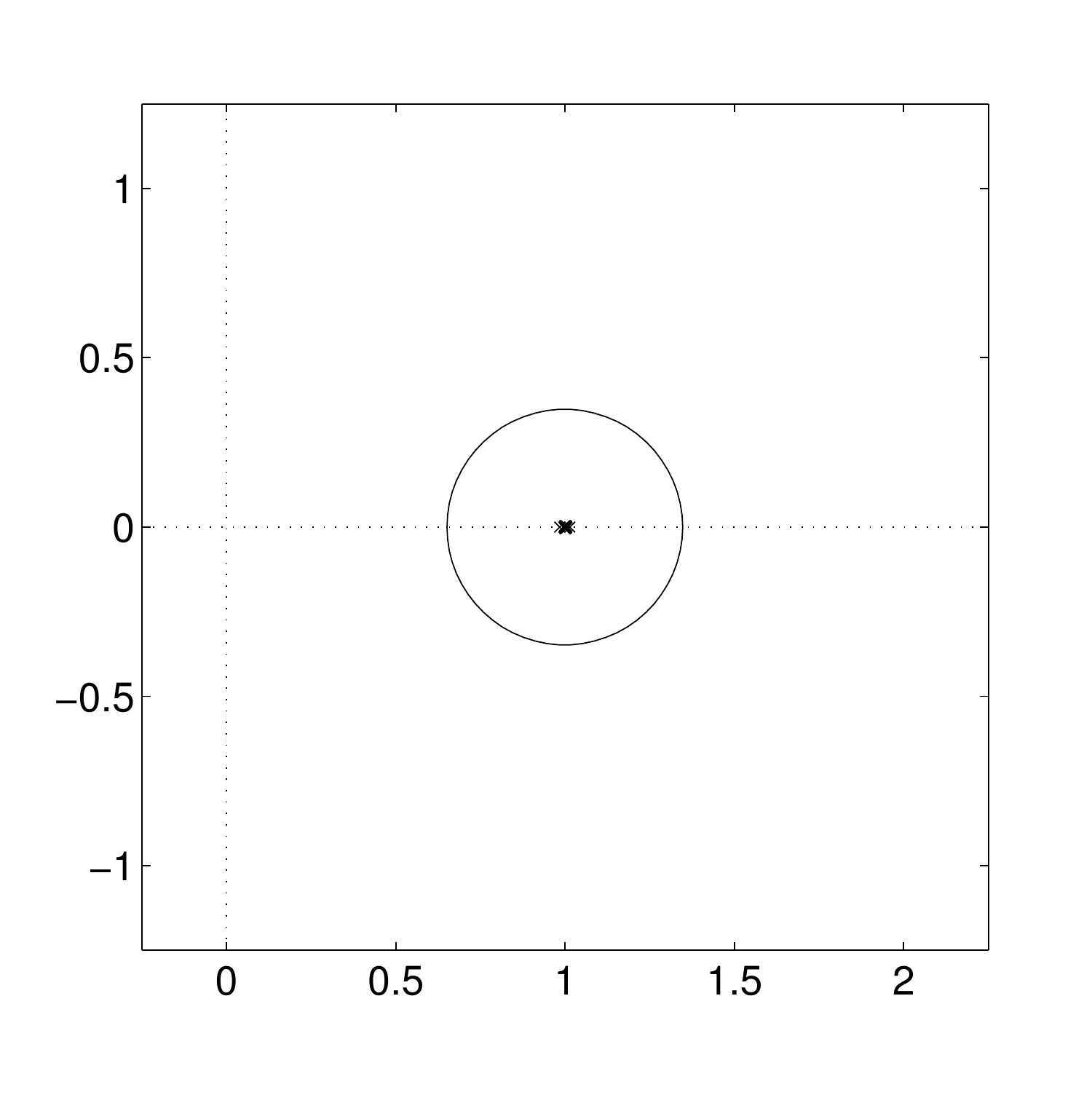}
}
\hspace{2.2cm}$\fov{B}$\hspace{3.2cm}$\fov{\widehat{B}}$\hspace{3cm}$\fov{\widehat{CB}}$
\caption{Spectrum and field of values of the matrices from three different test problems from section~7 before and after preconditioning by one multigrid V-cycle.}
\label{fig:preconditioned_spectrum}
\end{figure}

Figure~\ref{fig:preconditioned_spectrum} shows this comparison for three of the (smaller) examples of
section~\ref{sec:examples}. The first column depicts the field of values of the 
original matrix $B$, together with its eigenvalues. The second and third column 
give the same information for the projected matrix $\widehat{B}$ and the projected preconditioned matrix $\widehat{CB}$. The figure shows that the field of values almost 
remains unchanged when we go from $B$ to $\widehat{B}$. We note that $0$ must lie in the topological interior of $\fov{B}$, since if it were contained in the boundary, we would have 
$\nullsp(B) = \nullsp(B^T)$ by a result of \cite{Kippenhahn51}, see also \cite{HornJohnson:topics}, i.e.\ $x = \one.$ This would mean that $A$ is doubly stochastic which is not the case in our examples. An additional information, which due to the scaling cannot be seen in the figure, is that the field of values $\fov{\widehat{B}}$ still contains $0$ in its interior in all three examples, so that Theorem~\ref{starke:thm} cannot be used to predict a bound for the convergence speed of (non-preconditioned) GMRES.
For the preconditioned systems, $\fov{\widehat{CB}}$ is nicely bounded away from 0, and thus Theorem~\ref{sing_gmres2:thm} predicts a rapidly convergent method which is what we observe in our numerical experiments.

\section{Numerical examples}\label{sec:examples}
In this section, we report the results of numerical tests performed with the proposed method for a variety of Markov chains. The test problems are chosen such that many different scenarios occur: We consider (1) problems with a regular geometry as well as problems without any exploitable structure, (2) problems in which all transition probabilities are within the same order of magnitude as well as problems with highly varying probabilities,  and (3) structurally symmetric problems as well as nonsymmetric ones. For each test example, we report the number of iterations needed to reduce the scaled residual $\|Bx^{(k)}\|/\|x^{(k)}\|$ for the iterate $x^{(k)}$ to  $10^{-7}$, both for the BAMG preconditioned GMRES method and the GMRES method without preconditioning. For the latter, to limit memory and computational cost, we restart after every 50 iterations, i.e.\ we perform GMRES(50).
As a complement to the operator complexity $o_c$ from \eqref{operator_complexity_def:eq}, we also report the grid complexity $o_g = \frac{1}{n}\sum_{i = 1}^L n_i$.
These measures are commonly used in the multigrid literature to give an indication of how efficiently one cycle of the multigrid method can be executed. The lower the complexities, the more efficient a single cycle will be. The standard parameters for the method are three pre- and post-smoothing iterations, which is always the weighted Jacobi iteration with $\omega = 0.7$,  and $r = 8$ test vectors, unless explicitly stated otherwise. Most of the test cases are taken either from~\cite{BoltenBrandtBrannickFrommerKahlLivshits2011} or from~\cite{DeSterckManteuffelMcCormickRugeMillerPearsonSanders2010}. 

\begin{figure}
\centering
\includegraphics[width = 0.8\textwidth]{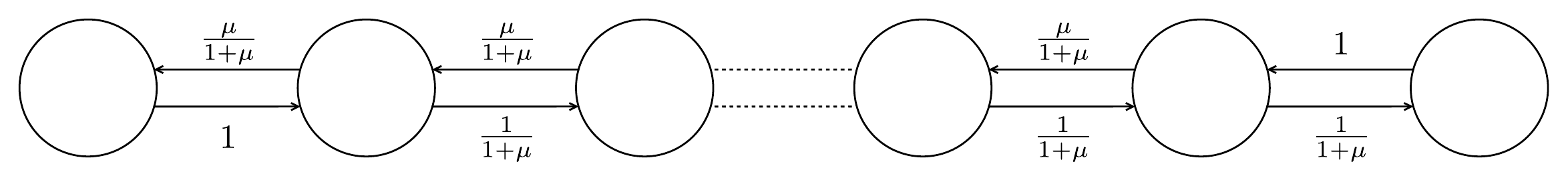}
\caption{One-dimensional birth-death chain}
\label{fig:birthdeath}
\end{figure}

\begin{table}
\centering
\footnotesize
\begin{tabular}{|c|c|c|c|c|c|c|c|}
\hline
$n$ & 1025 & 2049 & 4097 & 8193 & 16385 & 32769 & 65537\\
\hline
\hline
GMRES(50) & $> 1000$ & $> 1000$ & $> 1000$ & $> 1000$ & $> 1000$ & $> 1000$ & $> 1000$\\
BAMG + GMRES & 3 & 3 & 3 & 4 & 4 & 4 & 5 \\
BAMG$^2$ + GMRES & 1 & 1 & 1 & 1 & 1 & 2 & 2 \\
\hline
\end{tabular}
\caption{Iteration numbers for the one-dimensional birth-death chain}
\label{tab:birthdeath}
\end{table}

The first test example, a birth-death chain, has a regular one-dimensional structure. From each node, only transitions to its left and right neighbor are possible as depicted in Figure~\ref{fig:birthdeath}, which leads to a tridiagonal matrix $B$. We choose the parameter $\mu = 0.96$ as in~\cite{BoltenBrandtBrannickFrommerKahlLivshits2011,DeSterckManteuffelMcCormickRugeMillerPearsonSanders2010}.

For this example geometric coarsening (i.e., choosing every other variable as a coarse variable) is performed (until a problem size $n_c < 30$ is reached) and interpolation is done only from direct neighbors (with at most $c = 2$ interpolatory points), resulting in grid and operator complexities that are slightly below 2. The required iteration numbers for this problem are shown in Table~\ref{tab:birthdeath}. In the first row, the required number of iterations of non-preconditioned GMRES is reported. 
The second and third rows contain results obtained by performing one or two BAMG setup cycles, respectively, and then using the fixed hierarchy as a preconditionener for GMRES. The given integer value indicates the number of GMRES iterations needed in addition to the BAMG setup cycles to achieve the prescribed stopping tolerance, i.e., 1 means that one or two setup cycles and one iteration of preconditioned GMRES we performed. The iteration numbers clearly show that preconditioning improves the convergence behavior of GMRES substantially and that the method scales for the given range of problem sizes. The use of two setup cycles gives another noticeable improvement over just one cycle. This is to be expected since for two setup cycles the transfer operators are constructed with respect to the approximate singular vectors from the previous cycle, as opposed to just smoothed random vectors when only one cycle is used.

Due to the different transition probabilities to the left and to the right, the solution vector $x$ contains entries with highly varying order of magnitude, ranging from $\mathcal{O}(\mu^n)$ on the left end of the chain to $\mathcal{O}(1)$ on the right end. This is one reason why this test example is difficult to solve for some methods, as can be seen from the results for the non-preconditioned GMRES method in Table~\ref{tab:birthdeath}.  For the method
considered in~\cite{BoltenBrandtBrannickFrommerKahlLivshits2011} that, while the residual can be reduced below the prescribed tolerance in a few iterations, the computed solution contains small negative entries, which should not be the case. In the preconditioned method proposed here, we do not observe this behavior. 

\begin{figure}[t]
\centering
\subfloat[][]{
\begin{minipage}[c][1\width]{0.5\textwidth}
\centering
\includegraphics[width = 0.65\textwidth]{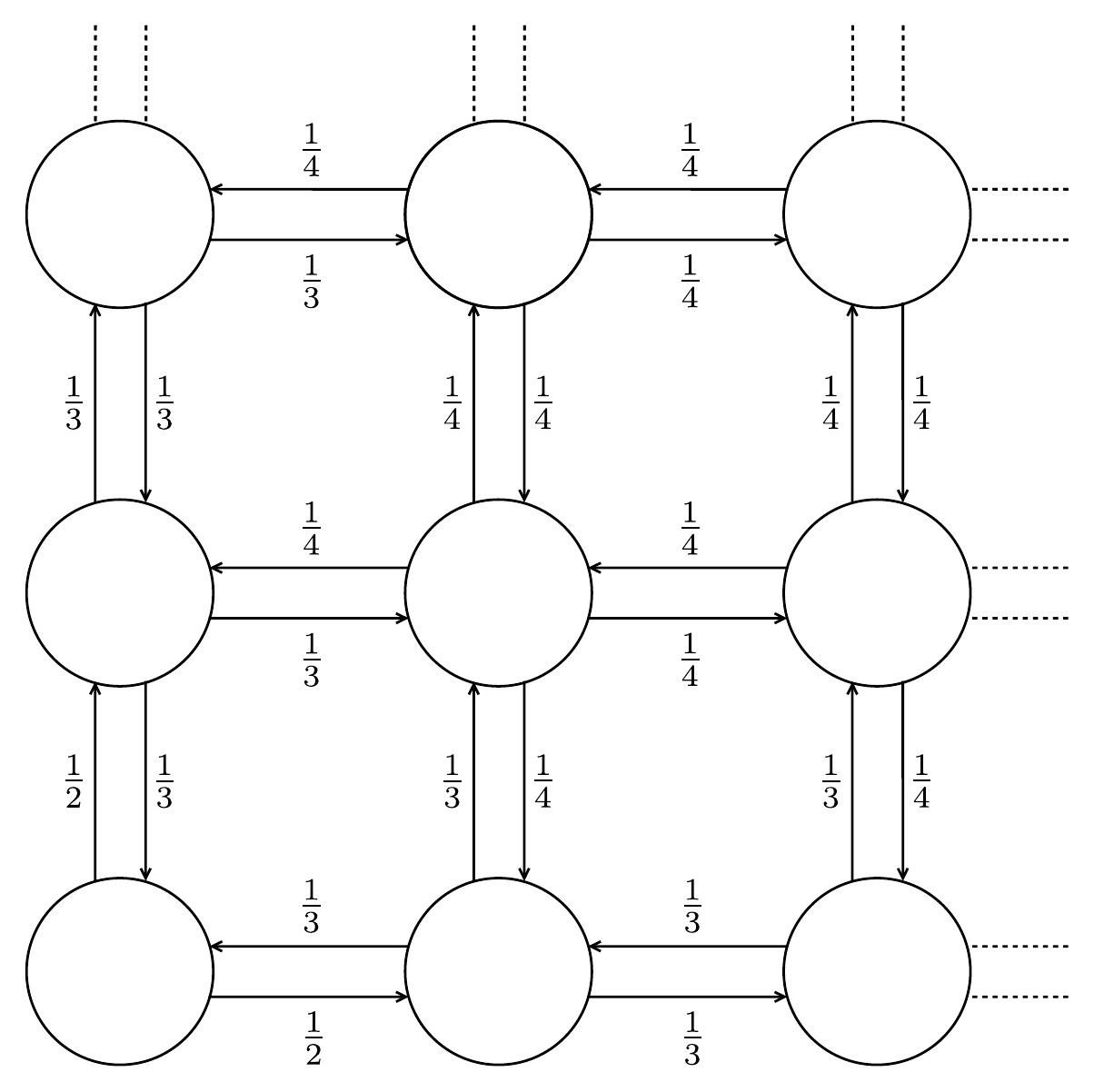}
\label{fig:uniform2d}
\end{minipage}
}
\subfloat[][]{
\begin{minipage}[c][1\width]{0.5\textwidth}
\centering
\includegraphics[width = 0.65\textwidth]{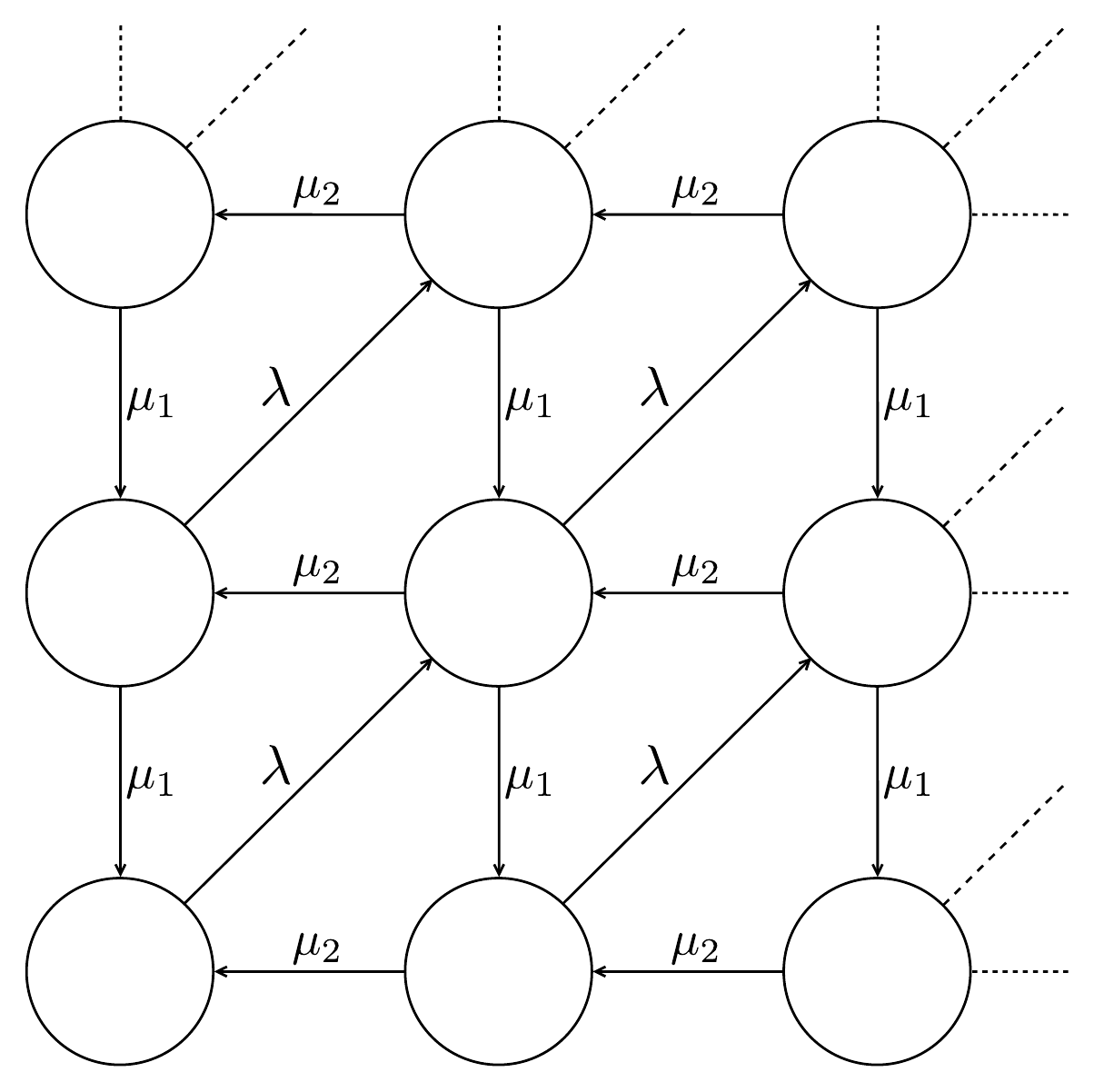}
\label{fig:tandem}
\end{minipage}
}
\caption{Uniform two-dimensional network (left) and two-dimensional tandem queuing network}
\end{figure}

\begin{table}
\centering
\footnotesize
\begin{tabular}{|c|c|c|c|c|c|c|}
\hline
$n$ & 1089 & 4225 & 9409 & 16641 & 25921 & 66049\\
\hline
\hline
GMRES(50)  & $52$ & $89$ & $105$ & $170$ & $200$ & $240$ \\
BAMG + GMRES & 6 & 6 & 6 & 7 & 8 & 10\\
BAMG$^2$ + GMRES & 4 & 4 & 4 & 5 & 5 & 6\\
\hline
\end{tabular}
\caption{Iteration numbers for the uniform two-dimensional network}
\label{tab:uniform2d}
\end{table} 

In the next two test examples, the Markov chains have a regular two-dimensional geometric structure. The first two-dimensional problem we consider is a uniform network, as depicted in Figure~\ref{fig:uniform2d}. All transition probabilities from one state are equal to the reciprocal of the number of neighboring states. We again use our standard parameters and  geometric coarsening. In the two-dimensional case this means that every other node in both directions is chosen as a coarse level node. We allow interpolation from up to $c = 4$ neighboring coarse level variables which leads to grid and operator complexities that are bounded by $1.4$ and $1.8$, respectively. Coarsening is done down to a problem size of $17^2$. The results obtained for this problem are reported in Table~\ref{tab:uniform2d}. 

As in the one-dimensional test case, we see that BAMG preconditioning needs far less iterations than non-preconditioned GMRES. The scaling behaviour of the BAMG preconditioned method is again very favorable, although we now need more than one iteration also in case of two setup cycles.

\begin{table}
\centering
\footnotesize
\begin{tabular}{|c|c|c|c|c|c|c|}
\hline
$n$ & 1089 & 4225 & 9409 & 16641 & 25921 & 66049 \\
\hline
\hline
GMRES(50) & $165$ & $295$ & $410$ & $570$ & $750$ & $1030$\\
BAMG + GMRES & 6 & 6 & 8 & 9 & 10 & 12\\
BAMG$^2$ + GMRES & 4 & 4 & 4 & 5 & 5 & 6\\
\hline
\end{tabular}
\caption{Iteration numbers for the tandem queuing network}
\label{tab:tandem}
\end{table}

The second test case with a regular two-dimensional geometric structure is the tandem queuing network example from~\cite{Stewart1994} which is depicted in Figure~\ref{fig:tandem}. The Markov chain considered here is the embedded discrete Markov chain of a continuous-time homogeneous Markov chain that describes a queuing system where each customer has to be served at two stations with one server each. In the continuous time case, the service time at station 1 and 2 is exponentially distributed with mean value $\dfrac{1}{\mu_1}$ and $\dfrac{1}{\mu_2}$, respectively, and customers arrive according to a Poisson process with rate $\lambda$. In the discrete case, this leads to transition probabilities $\mu_1,\mu_2$ and $\lambda$, as in Figure~\ref{fig:tandem}. The diagonal transitions describe the arrival of a new customer, the horizontal and vertical arrows describe that a customer has been served at station 1 or 2. We chose the parameters to be $\lambda = \frac{11}{31}, \mu_1 = \frac{10}{31}$ and $\mu_2 = \frac{10}{31}$, as in \cite{BoltenBrandtBrannickFrommerKahlLivshits2011,DeSterckManteuffelMcCormickRugeMillerPearsonSanders2010}. In contrast to the test cases so far, the matrix of the tandem queuing network is structurally nonsymmetric and has a complex spectrum. We test the method with the same parameters as we used for the uniform two-dimensional network and the results are given in Table~\ref{tab:tandem}. The behavior is almost the same as for the uniform two-dimensional network, although the problem is much harder to solve for the non-preconditioned GMRES iteration. This again underlines the high quality of the adaptively computed multigrid hierarchy.
 
\begin{figure}
\centering
\includegraphics[width = 0.9\textwidth]{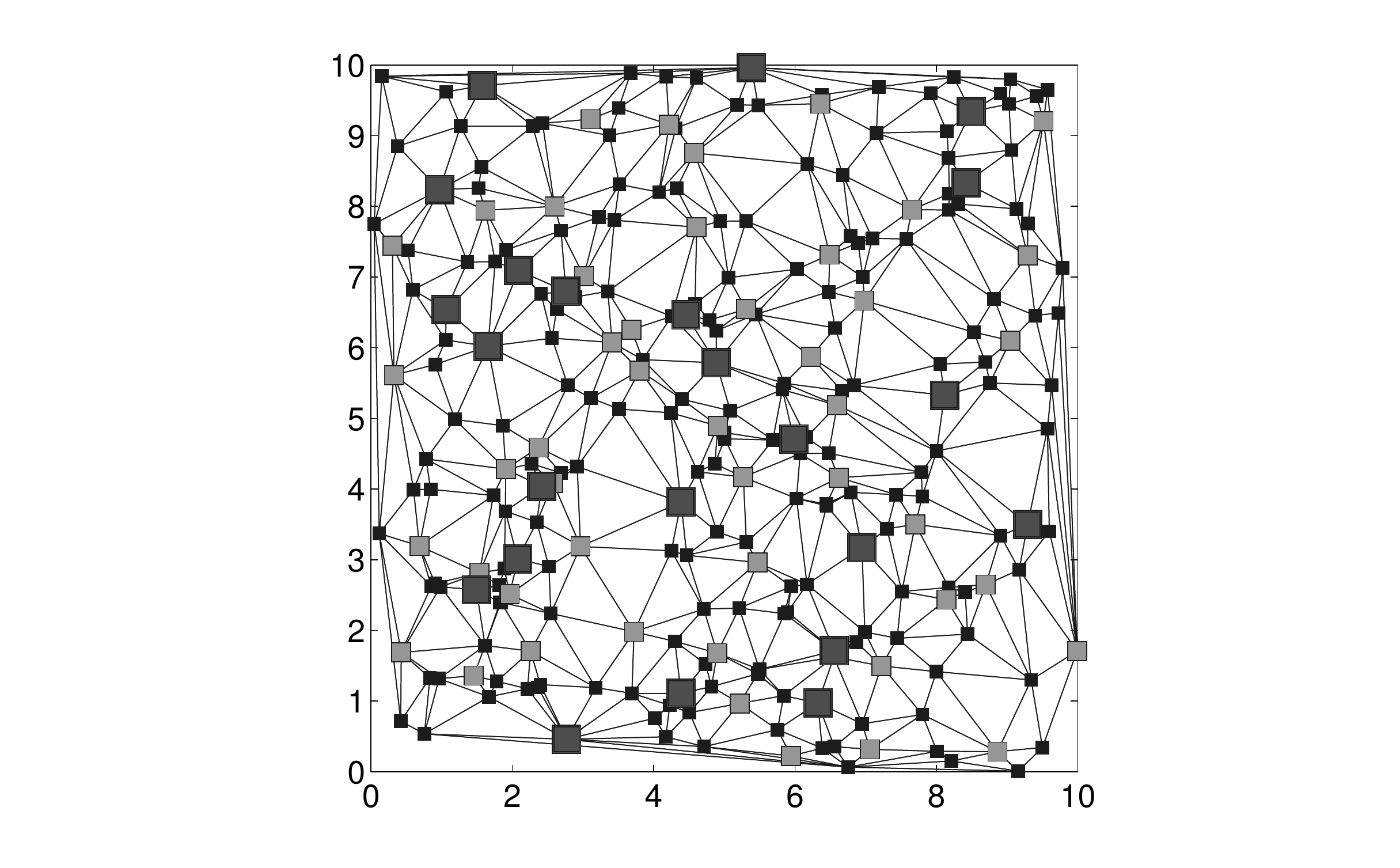}
\caption{Random planar graph with 256 nodes and coarsening by compatible relaxation. The smallest nodes are the initial grid, the medium size nodes are selected for the second level and the largest nodes for the third level}
\label{fig:planargraph}
\end{figure}

\begin{table}
\centering
\footnotesize
\begin{tabular}{|c|c|c|c|c|c|c|}
\hline
$n$ & 1024 & 4096 & 8192 & 16384 & 32768 & 65536\\
\hline
\hline
GMRES(50) & $50$ & $80$ & $105$ & $115$ & $135$ & $170$ \\
BAMG + GMRES & 8 (2)& 10 (2)& 12 (3)& 15 (3) & 19 (4) & 22 (4)\\
BAMG$^2$ + GMRES & 3 (2)& 3 (2)& 4 (3)& 5 (3) & 5 (4) & 6 (4)\\
\hline
\end{tabular}
\caption{Iteration numbers and number of levels for the random planar
graphs}
\label{tab:planargraph}
\end{table}

Our last two test examples do not exhibit any regular geometric structure, so that the coarsening has to be done algebraically. In the tests we use a compatible relaxation algorithm~\cite{BrannickFalgout2010}. We first consider random walks on planar graphs. These planar graphs are generated, as suggested in~\cite{DeSterckMillerSandersWinlaw2010}, by choosing $n$ random points in the unit square and then constructing a Delaunay triangulation of these points. Transitions between states are then restricted to  the edges of the triangles, with transition probabilities again chosen as the reciprocals of the number of outgoing edges for each node. For these problems, we use $r = 8$ test vectors, coarsen down to a grid size $n_c < 500$ and allow interpolation from up to $c = 3$ neighboring nodes. The number of pre- and postsmoothing iterations are both increased 
to 5. The results of the method are given in Table~\ref{tab:planargraph}. The numbers in brackets  report the number of levels used to reach a problem size of less than 100 variables on the coarsest level. This results in grid and operator complexities which are bounded by 1.3 and 1.5, respectively, for all considered problem sizes.
The results show a significant improvement compared to the eigenvector based BAMG approach in~\cite{BoltenBrandtBrannickFrommerKahlLivshits2011}, where a substantial increase in iteration numbers was observed for increasing problem sizes. With the proposed singular vector based BAMG approach we observe now only a mild dependence on the number of iterations on the problem size, especially when using two setup cycles.

\begin{figure}
\centering
\includegraphics[width = 0.25\textwidth]{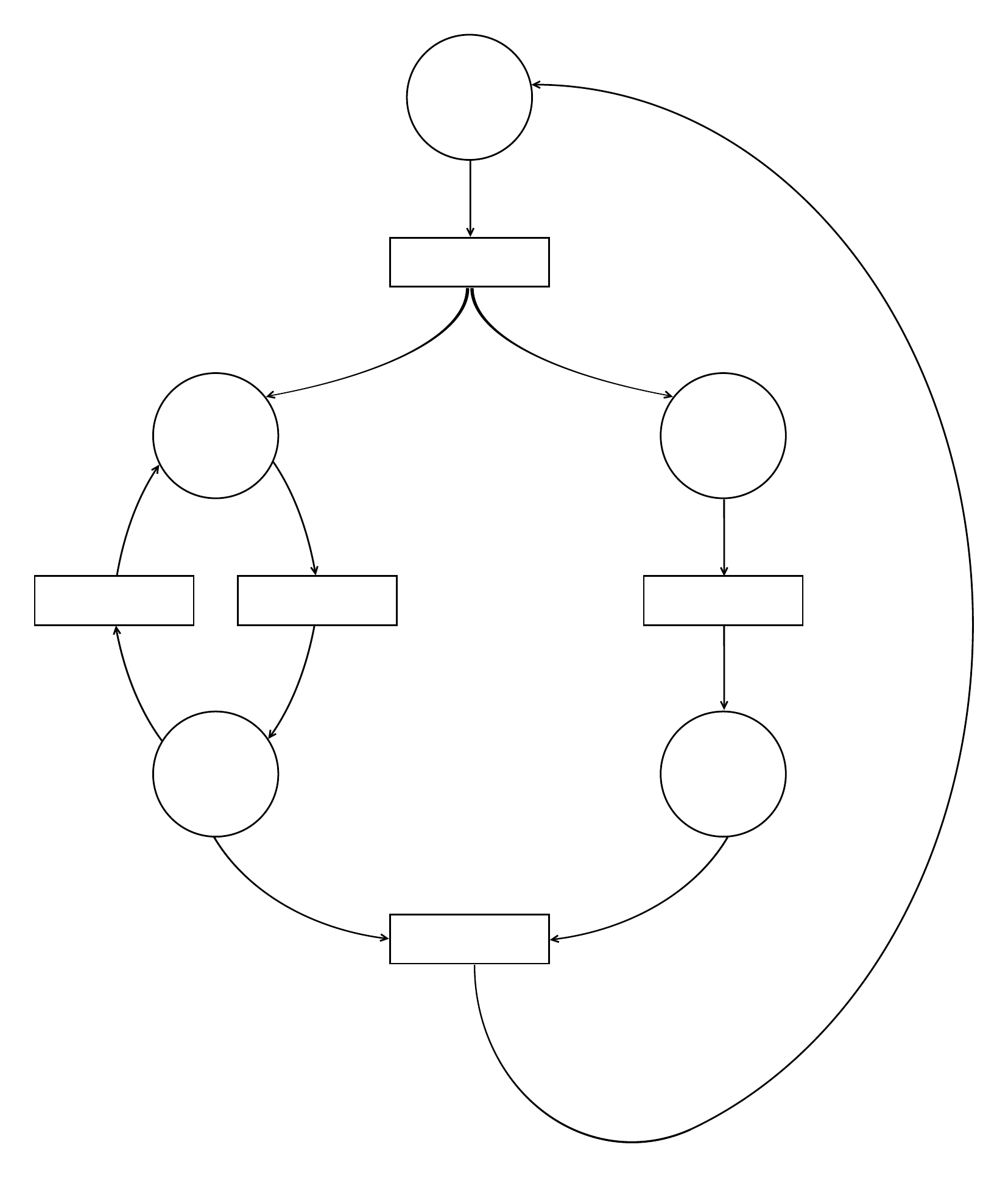}
\caption{Petri net with five places and five transitions}
\label{fig:petrinet}
\end{figure}

\begin{table}[t]
\centering
\footnotesize
\begin{tabular}{|c|c|c|c|c|}
\hline
$n$ & 1496 & 3311 & 10416 & 23821\\
\hline
\hline
GMRES(50) & $72$ & $90$ & $95$ & $180$\\
BAMG + GMRES & 6 & 6 & 6 & 6\\
BAMG$^2$ + GMRES & 5 & 5 & 5 & 6 \\
\hline
\end{tabular}
\caption{Iteration numbers for the Petri net}
\label{tab:petrinet}
\end{table}

The last test example is a Petri net taken from~\cite{Molloy1982}. It was, e.g., also considered in~\cite{DeSterckMillerSanders2011} and~\cite{HortonLeutenegger1994}. Petri nets are used for the description of distributed systems and consist of places and transitions. In the simple Petri net depicted in Figure~\ref{fig:petrinet}, the circles represent places and the bars represent transitions. The places can be seen as conditions under which certain transitions can occur: A transition with ingoing edges from two places can only occur when both places are marked by at least one {\em token}. When the transitions occur, the tokens will disappear and in each place reached by an outgoing edge of the transition, a new token will appear. This way, starting from an initial marking, the tokens will travel through the different places of the Petri net. A Petri net can be transformed into a discrete Markov chain as follows (cf.~\cite{MotameniMovagharSiasifarMontazeriRezaei2008}): Each marking of the Petri net that is reachable from a given initial marking corresponds to one state of the Markov chain. Transitions occur between all markings that can be converted into each other by the firing of one transition. The transition probabilities can then be calculated from the known firing probabilities of the transitions of the Petri net. The size of the resulting Markov chain depends on the number of tokens in the initial marking and grows very fast. When the initial marking consists of only one token in the uppermost node in the Petri net from Figure~\ref{fig:petrinet}, the Markov chain will have only 5 states, for 10 tokens in the same node the Markov chain has 506 states, and for 30 tokens it already consists of 10416 states, although the underlying Petri net has only 5 places and 5 transitions. The transition matrix of the resulting Markov chain is structurally nonsymmetric with complex spectrum, as depicted in Figure~\ref{fig:preconditioned_spectrum}. We chose the firing probabilities of the transitions to be the same as in~\cite{DeSterckMillerSanders2011} and~\cite{HortonLeutenegger1994}. Since there is no underlying geometric structure, coarsening is again done by compatible relaxation. We use $r = 10$ test vectors and the other parameters of the algorithm are the same as for the planar graph problems. The results of the tests are reported in Table~\ref{tab:petrinet} and the grid and operator complexities are bounded by 1.7 and 2.5, respectively.. The problem sizes correspond to a Petri net with an initial marking of 15, 20, 30 and 40 tokens in the uppermost node, respectively. While the two versions of the preconditioned GMRES iteration scale for the problem sizes considered, we see no significant reduction in the iteration numbers when using two setup cycles instead of one.

\section{Conclusions}\label{sec:conclusion}

In this paper, we developed a bootstrap algebraic multigrid framework for non-symmetric matrices.  We demonstrated 
that using singular vectors instead of eigenvectors in the adaptive construction of transfer operators leads to a
new, more robust approach for the targeted Markov problems. The approach is shown to
be especially suitable in the context of Markov chains because natural assumptions
on the coarse grid matrix, e.g., singularity and column sum zero,
can automatically be fulfilled when using the non-symmetric bootstrap setup. We showed that this can be accomplished by modifying the least squares interpolation to allow for linear constraints.



We experimentally explained the fast convergence of the preconditioned system by relating the GMRES iteration for the singular system to the one for an equivalent, projected non-singular system with a system matrix that has its field of values well separated from 0. Standard results on the speed 
of convergence of GMRES in the non-singular case can then be used to explain the fast convergence of the method.

The computational examples show the superiority of the singular vector based method over the 
eigenvector based method from~\cite{BoltenBrandtBrannickFrommerKahlLivshits2011}. The method 
further reduces the iteration count and yields near-optimal scaling behavior for the considered 
test problems, especially when two setup cycles are used, therefore taking full advantage of the 
improved test vectors. Thus we expect the new method to be particularly efficient for large problem sizes. When comparing computational cost in our numerical examples one should be aware that one iteration with the BAMG preconditioner is roughly equivalent to $(s_{pre}+s_{post})o_c$ 
matrix-vector multiplications, where $s_{pre}$ and $s_{post}$ is the number of pre- and post-smoothing iterations, respectively, and $o_c$ is the operator complexity. For each BAMG setup cycle, this cost has to be multiplied by $2r$, $r$ being the number of test vectors in $\mathcal{U}$ and $\mathcal{V}$.
Measured in this manner, our numerical examples indicate that for the still quite moderate problem 
sizes considered there we already have a substantial gain of the new BAMG preconditioned method over non-preconditioned GMRES for the more difficult problems, i.e.\ the death-birth chain, the tandem queueing network and the Petri net.

\bibliographystyle{elsarticle-num}
\bibliography{markov}{}

\end{document}